\documentclass[11pt,letterpaper]{article}

\usepackage[left=28mm, right=28mm, top=30mm, bottom=35mm]{geometry}
\usepackage{amsmath,amsfonts,amsthm,amssymb}
\usepackage[abbrev,lite,nobysame]{amsrefs}
\usepackage{mathtools}

\usepackage{enumerate}

\mathtoolsset{showonlyrefs}
\usepackage{tikz,tikzscale}
\usepackage{hyperref,color}
\usepackage[capitalize]{cleveref}
\usepackage{aliascnt}

\newtheorem{theorem}{Theorem}[section]
\numberwithin{equation}{section}

\newaliascnt{proposition}{theorem}
\newaliascnt{lemma}{theorem}
\newaliascnt{corollary}{theorem}
\newaliascnt{remark}{theorem}

\newtheorem{proposition}[proposition]{Proposition}
\newtheorem{lemma}[lemma]{Lemma}
\newtheorem{corollary}[corollary]{Corollary}
\newtheorem{remark}[remark]{Remark}

\aliascntresetthe{proposition}
\aliascntresetthe{lemma}
\aliascntresetthe{corollary}
\aliascntresetthe{remark}

\crefname{appendix}{Appendix}{Appendices}

\newcommand{\mf}{\mathfrak}  

\newcommand{\R}{\mathbb{R}}
\newcommand{\C}{\mathbb{C}}
\newcommand{\Z}{\mathbb{Z}}
\newcommand{\X}{\mathbb{X}}
\newcommand{\N}{\mathbb{N}}
\newcommand{\T}{\mathbb{T}}
\renewcommand{\P}{\mathbb{P}}
\newcommand{\Q}{\mathbb{Q}}
\newcommand{\J}{\mathbb{J}}
\newcommand{\I}{\mathbb{I}}

\newcommand{\cA}{\mathcal A}
\newcommand{\cS}{\mathcal S}
\newcommand{\cL}{\mathcal L}
\newcommand{\cT}{\mathcal T}
\newcommand{\cK}{\mathcal K}
\newcommand{\cJ}{\mathcal J}
\newcommand{\cH}{\mathcal H}
\newcommand{\cR}{\mathcal R}
\newcommand{\cX}{\mathcal X}
\newcommand{\cY}{\mathcal Y}
\newcommand{\cO}{\mathcal O}
\newcommand{\cB}{\mathcal B}
\newcommand{\cF}{\mathcal F}
\newcommand{\HX}{\cH _\cX}

\newcommand{\e}{\varepsilon}
\newcommand{\om}{\omega}

\newcommand{\cLe}{\cL _\e}
\newcommand{\cTe}{\cT _\e}
\newcommand{\cKe}{\cK _\e}

\newcommand{\cHe}{\cH _\e}
\newcommand{\ue}{u _\e}
\newcommand{\ome}{\omega _\e}
\newcommand{\fe}{\phi _\e}
\newcommand{\vfe}{\varrho _\e}
\newcommand{\ye}{\psi _\e}
\newcommand{\Fe}{F _\e}
\newcommand{\Ae}{A _\e}
\newcommand{\Be}{B _\e}
\newcommand{\Xe}{X _\e}
\newcommand{\Ye}{Y _\e}
\newcommand{\Phie}{\Phi _\e}
\newcommand{\Pe}{\P _\e}
\newcommand{\Qe}{\Q _\e}
\newcommand{\lae}{\lambda _\e}

\newcommand{\unu}{u ^\nu}
\newcommand{\omnu}{\omega ^\nu}
\newcommand{\pnu}{p ^\nu}

\newcommand{\uE}{u ^{\mathrm E}}
\newcommand{\omE}{\omega ^{\mathrm E}}
\newcommand{\pE}{p ^{\mathrm E}}
\newcommand{\yE}{\psi ^{\mathrm E}}
\newcommand{\CE}{C ^{\mathrm E}}
\newcommand{\cE}{c ^{\mathrm E}}
\newcommand{\dE}{\delta ^{\mathrm E}}

\renewcommand{\d}{\mathop{\kern0pt\text{\upshape d}}\!{}}
\renewcommand{\div}{\operatorname{div}}
\DeclareMathOperator{\curl}{curl}
\DeclareMathOperator{\Lip}{Lip}
\newcommand{\expo}[1]{\operatorname e^{#1}}
\renewcommand{\Re}{\operatorname{Re}}

\title{Fast separation of Euler and Navier--Stokes flows in~two~dimensions}

\author{Siming He\thanks{Department of Mathematics, University of South Carolina. Email: \url{siming@mailbox.sc.edu}} , 
Yupei Huang\thanks{Department of Mathematics, Imperial College London. Email: \url{yhuang9@ic.ac.uk}} , 
Jincheng Yang\thanks{Department of Applied Mathematics and Statistics, Johns Hopkins University. Email: \url{jincheng@jhu.edu}}}

\date{} 
\begin{document}

\maketitle

\begin{abstract}
We study the long-time validity of the inviscid limit problem by measuring when the two-dimensional ideal fluid ceases to accurately approximate a slightly viscous fluid with the same initial datum. We construct a stationary solution of the Euler equation from which the Navier--Stokes solution separates at a time scale proportional to the logarithm of the inverse viscosity, achieving the fastest separation from Euler steady states.
\end{abstract}


\par
\par\medskip

\tableofcontents

\section{Introduction}\label{sec:Introduction} 
The 2D Navier--Stokes equation is a fundamental fluid model to describe the motion of a viscous fluid. On the torus $\T ^2$, the equation takes the form  
\begin{equation}
    \label{eq:ns}
    \partial_t \unu+\unu\cdot\nabla \unu+\nabla \pnu
    =\nu\Delta \unu,\quad 
    \div \unu=0,\quad
    \unu(0,x)=f(x),
\end{equation}
where $\unu$ is the velocity, $\pnu$ is the pressure, and $\nu$ is the kinematic viscosity. The model reflects the interaction between nonlinear advection and viscous diffusion. The corresponding inviscid model, obtained by setting the kinematic viscosity to zero, is the incompressible Euler equation:
\begin{equation}
    \label{eq:euler}
    \partial_t \uE+\uE\cdot\nabla \uE+\nabla \pE=0,\quad
    \div \uE=0,\quad 
    \uE(0,x)=f(x).
\end{equation}
When the viscous effect is negligible, physicists treat the incompressible Euler equations as an idealized approximation to the Navier--Stokes dynamics; see, e.g., \cites{Batchelor1967,Acheson1990,LandauLifshitz1987}. A fundamental mathematical question is to characterize the functional framework and time scales on which the Navier--Stokes solutions $\unu$ are well-approximated by the Euler flow $\uE$ as the kinematic viscosity approaches zero. In particular, we are interested in whether a fixed smooth Euler steady state can generate a fixed positive $L^2$ discrepancy on a time scale of order $\log(\nu^{-1})$ when used as the common initial velocity for \eqref{eq:ns} and \eqref{eq:euler}.

\subsection{The validity of inviscid approximation}

Over a fixed time horizon, inviscid limit is valid for sufficiently regular initial data. In the work of \cites{Constantin1988navier,Masmoudi2007remarks}, by assuming the initial data is the smooth function $f$, the authors establish that the Navier--Stokes solution $\unu$ converges to the Euler solution $\uE$ in the smooth category on any $\nu$-independent finite time interval. These inviscid limit results have also been extended to less regular initial data. When the initial vorticity  lies in the Yudovich class, Constantin, Drivas, and Elgindi \cite{ConstantinDrivasElgindi2022} proved strong convergence of the velocity on $\T ^2$ in every finite $L^p$ space, uniformly on fixed finite time intervals.  

Besides the inviscid limit problem in fixed time, people also  derived qualitative bound for velocity difference global in time.  In the whole-space setting, Chemin \cite{Chemin1996} established an $L^2$ velocity difference estimate in the Yudovich class, subsequently refined by Seis \cite{Seis2021}. Related velocity estimates on the flat torus are included in \cite{ConstantinDrivasElgindi2022}. Those works imply that on the time scale of $o (\log \log \nu^{-1})$, the solutions to the Navier--Stokes equation will converge uniformly to the solution of the Euler equation. Such results provide sufficient conditions for the solutions to remain close, but do not identify the time scale when a
fixed positive discrepancy actually develops.

Our aim is to characterize how viscosity and the structure of the initial datum affect the separation time scale
\begin{equation}\label{eq:sep-time}
    T^{\delta} _\nu(f) := \inf \left\{ t \ge 0: \| \unu(t)-\uE(t) \|_{L ^2 (\T ^2)} \geq \delta \| f \| _{L ^2 (\T ^2)} \right\}.
\end{equation}  
Here $\delta > 0$ is a chosen separation level. Due to the dissipation effect, it is clear that in the diffusive time scale, which is proportional to $\nu^{-1}$, the solution of the Navier--Stokes equation will separate from the corresponding Euler solution. As a matter of fact, if the initial data is a shear flow or the stream function corresponds to an eigenfunction for the Laplacian on the torus, the Navier--Stokes equation behaves exactly as the heat equation. In this case, the separation time scale is exactly proportional to the diffusive time scale. As we mentioned before, the work \cite{ConstantinDrivasElgindi2022} demonstrates the separation time scale should be at least the size of $\cO (\log \log \nu^{-1})$. To date, there is no explicit example in which the separation time scale is substantially shorter than the diffusion time scale $\nu ^{-1}$, let alone one that attains the $\cO (\log \log \nu^{-1})$ lower bound. 

It is an important task to construct examples exhibiting a fast Euler-Navier--Stokes separation. In the vanishing viscosity regime, this mission is naturally linked to various 2D Euler instabilities   \cites{KiselevSverak2014,FriedlanderStraussVishik1997,FriedlanderStraussVishik1999}. Our goal is to demonstrate that the instability of Euler steady states could trigger the desired fast Euler-Navier--Stokes separation. To this end, we present the necessary background in the next subsection.

\subsection{A fast separation scenario}
In this section, we introduce the  separation mechanism. Let us consider the vorticity form of Navier--Stokes and Euler equation, which govern the dynamics of $\omnu : = \nabla ^\perp \cdot \unu=- \partial_y \unu _1 + \partial_{x} \unu _2$ and $\omE := \nabla ^\perp \cdot \uE,$
\begin{align}
    \label{eq:NS_vor}
    \partial _t \omnu + \unu \cdot \nabla \omnu &= \nu \Delta \omnu, &
    \partial_t \omE+\uE\cdot\nabla \omE &= 0. 
\end{align}
Throughout this paper, we consider  steady state $\uE$ of the Euler equation: 
\begin{align}
    \label{eq:steady-Euler}
    \uE \cdot \nabla \omE = 0.
\end{align}
To study the separation time $T_\nu^\delta$ \eqref{eq:sep-time}, it is natural to consider the velocity and vorticity differences $v := \unu - \uE$, $w := \omnu - \omE = \nabla ^\perp \cdot v$, governed by the dynamics
\begin{equation}\label{eq:w_cpt}
    \begin{split}&\partial _t w-\cL w  = \nu \Delta w+ \nu \Delta \omE - \curl^{-1} w \cdot \nabla w, \quad w(t=0)\equiv 0,\\
    &\cL w:= \cT w + \cK w := - \uE \cdot \nabla w - \nabla \omE \cdot \curl^{-1} w.
    \end{split}
\end{equation}
Here $\curl ^{-1} := \nabla ^\perp \Delta ^{-1}$ is the Biot--Savart operator and $\cL$ is the linearized Euler operator.

The static external force $\nu\Delta\omE$ drives the growth in \eqref{eq:w_cpt}. To maximize the growth rate, we seek a scenario where this forcing triggers the most unstable mode of $\cL$. An extensive literature uses the most unstable modes of $\cL$ to study the nonlinear instability of the 2D Euler equations; see, for example, \cites{BelenkayaFriedlanderYudovich1999,Lin2003,bardos2002stable,lin2004nonlinear,lin2013unstable,Vishik2018PartI,Vishik2018PartII,AlbrittonEtAl2024Vishik}. The classical strategy is organized in two stages. In the first stage, one identifies an Euler steady state whose linearized operator has an unstable mode. In the second stage, one introduces carefully tuned initial perturbation or forcing that trigger the unstable mode of the \emph{given} linearized Euler operator and invokes the so-called bootstrap lemma to close the nonlinear argument. In this approach, the choice of the perturbation and that of the underlying Euler steady state are \emph{decoupled}. 

In our separation scenario, the forcing $\nu \Delta \omE$ and the Euler steady state $\uE$ are \emph{coupled}. Hence, the existence of an unstable mode alone is insufficient for fast separation, because we also need to ensure that the forcing has a nontrivial spectral projection onto the unstable modes. It turns out that for most steady states in the literature, the corresponding forcing $\nu \Delta \omE$ has a trivial projection onto the unstable modes. A major challenge is to design a steady state and a function space in which the forcing has a nontrivial projection. Our innovation is to introduce a symmetry class preserved by the Euler/Navier--Stokes evolution, within which we find a steady state for which the forcing is guaranteed to have a nontrivial projection onto the unstable modes.

\subsection{Main result}
We define a symmetry class
\begin{align}
    \label{def:X}
    \cX &= \left\{ 
        f \in L ^2 _0 (\T ^2) \mid f (x, y) = f (-x, -y) = f \left(x + \frac{2 \pi} 3, y\right) = f \left( \frac{2 \pi}{15} - x, \frac{2 \pi}{5} - y \right)
    \right\}.
\end{align}
Here $L ^2 _0 (\T ^2)$ represents the space of $L ^2$ function with mean zero. An illustration of this class is given in \cref{fig:symmetry}.
A function $f \in \cX$ has the following Fourier series expansion:
\begin{align}
    \label{eqn: fourier}
    f = \sum _{\substack{j \ge 0 \\ n \in \Z \\ 5 \mid j + n}} c _{j, n} \cos (3 j x + n y), \qquad \sum _{j, n} |c _{j, n}| ^2 < \infty.
\end{align}
To avoid repetition, we enforce $c _{0, n} = 0$ for $n < 0$, and mean zero gives $c_{0,0}=0$.

\begin{theorem}
    \label{thm:main}
    There exists a non-shear stationary Euler solution $\uE$ near $(5 \sin{5y}, 0) ^\top$ in the smooth category such that the vorticity $\nabla^{\perp}\cdot\uE $ is in the symmetry class $ \cX$. Moreover, there exist $\uE$-dependent parameters $\CE \geq 1$, $\dE, \cE \in (0, \expo{-1})$ such that for all $0<\delta\leq \dE$ and all $0<\nu \le \cE \delta$, the following estimate holds 
    \begin{align*}
        T ^\delta _\nu (\uE) \le \CE \log \frac \delta \nu.
    \end{align*}
\end{theorem}

\begin{figure}[!h]
    \centering
    \begin{tikzpicture}[scale=2.5]
        \node (0, 0) {\includegraphics[width=0.3155\textwidth]{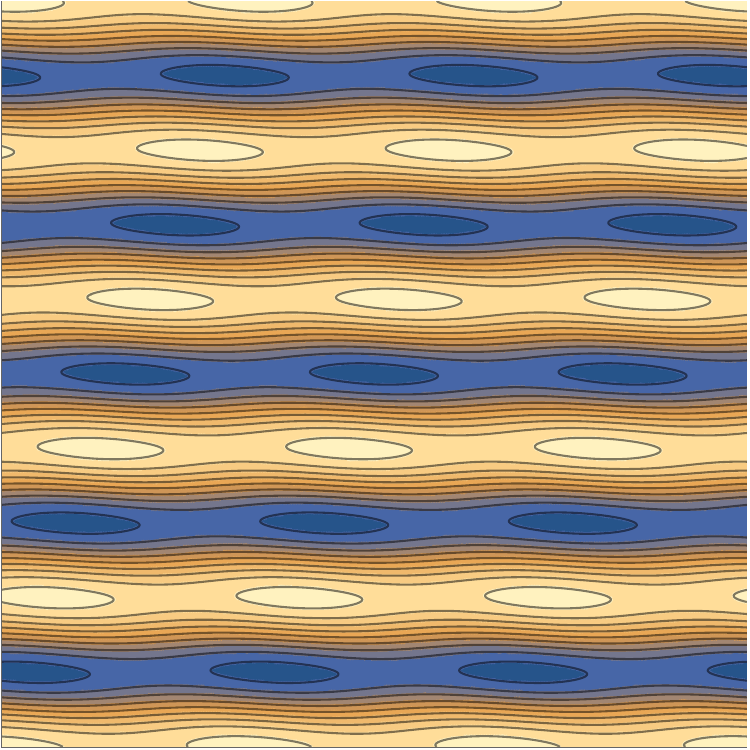}};
        \draw[red, thick] (-1, -1) -- (-1, 1) -- (1, 1) -- (1, -1) -- (-1, -1);
        \draw[red, thick] (-1/3, -1) -- (-1/3, 1);
        \draw[red, thick] (1/3, -1) -- (1/3, 1);
        \draw[green,ultra thick,dashed] (0, -1) -- (-4/3, -1) -- (-2/3, 1) -- (0, 1) -- (-2/3, -1);
        \draw[green,ultra thick,dashed] (0, 1) -- (4/3, 1) -- (2/3, -1) -- (0, -1) -- (2/3, 1);
        \fill[blue, opacity=0.4] (-2/3, -1) -- (-1.6/3, -0.6) -- (0.4/3, -0.6) -- (0, -1);

        \fill (-1/3, -1) circle (0.02) node[below] {$O$};
        \fill[white] (-0.8/3, -0.8) circle (0.02);
    \end{tikzpicture}
    \caption{the vorticity $\omE$ and the symmetry class of a function $f \in \cX$. Each red rectangle is $\frac{2\pi}3$ by $2 \pi$, and the periodicity condition in $x$ requires that $f$ repeats itself in each green dashed parallelogram. Central symmetry around the origin $O$ enforces the graph inside each parallelogram to be centrally symmetric, and central symmetry around $(\frac\pi{15}, \frac\pi5)$ (marked by the white dot) implies $f$ is centrally symmetric in the shadowed blue parallelogram, obtained by dividing the green one evenly into five pieces; and $f$ repeats itself in these pieces. The background contour plot is our construction of the vorticity $\omE$.}
    \label{fig:symmetry}
\end{figure}

\begin{remark}
     We prove in the Appendix \ref{sec:lower-bound} for every smooth stationary Euler flow $\uE$, there exist constants $\cE$, $\CE$ depending on $\uE$ with
    \begin{align}
        T _\nu ^\delta (\uE) \ge \cE \log \left(\frac{\delta^2}\nu\right) - \CE,
    \end{align}
    hence the separation time scale for the state $\uE$ in Theorem \ref{thm:main} is exactly the same order as $\log (1/\nu)$.
\end{remark}

\begin{remark}
    
    In the case of a bounded domain, the unconditional inviscid limit is valid for the Navier boundary condition \cites{filho2005,Kelliher2006Navier,Lions1996}, but is in general unknown for the no-slip boundary condition \cites{Bardos2013,maekawa2018}, with possible exceptions including additional assumptions on energy dissipation \cites{Kato1984,Kelliher2007Kato}, analyticity \cites{SammartinoCaflisch1998II}, or symmetry \cites{BonaWu2002,MazzucatoTaylor2008,GieEtAl2019}. See also \cites{CaflischSammartino2006,ConstantinKukavicaVicol2015,GerardVaretMaekawaMasmoudi2018,GerardVaretNguyen2012,GrenierGuoNguyen2016,KukavicaVicolWang2017,LacaveMazzucato2016,Maekawa2014,BedrossianHeIyerWang2024} and the references therein. Nevertheless, it is still possible to establish a lower bound on the separation time. Following the separation estimate in \cites{vasseur2023,vasseur2024}, one can show that $T _\nu ^\delta (\uE) \ge \delta ^2 / C \| \uE \| _{L^\infty}$ as $\nu \to 0$ in periodic channels, and a similar lower bound for a general smooth domain.
\end{remark}

In the remainder of the introduction, we present our general non-shear construction and motivate the introduction of the symmetry class $\cX$. We start by choosing a shear flow for which we can explicitly calculate the most unstable modes. The sinusoidal flow serves as a natural candidate. The results in \cites{MeshalkinSinai1961,BelenkayaFriedlanderYudovich1999} enable us to apply Fourier expansions and continued-fraction techniques to describe the most unstable modes. One can hope to construct non-shear steady states near this sinusoidal flow based on perturbation methods detailed in \cites{coti2023stationary,lin2011inviscid,elgindi2026flexibility}. For instance, we can get a family of non-shear steady states near $(5 \sin{5y}, 0) ^\top$ by rescaling the steady states in \cite{coti2023stationary} near the Kolmogorov flow. However, one can check that the forcing $\nu\Delta\omE$ has trivial projection onto the most unstable mode (see \cref{rem:rescale-bad}). Similarly, other methods fail to provide the desired steady states directly. In contrast, if we apply the method in \cite{coti2023stationary} in the smaller symmetry class $\cX\subsetneq L^2$, the most unstable mode is fundamentally different, and the forcing has nontrivial projection onto this mode.

Our paper is organized as follows. In \cref{sec:Preliminary}, we outline the proof strategy and state the main steps as propositions that lead to the proof of \cref{thm:main} at the end of the section. In \cref{sec:kolmogorov}, we perform a spectral analysis of the sinusoidal shear flow within the symmetry class $\cX$. An explicit stationary solution $\uE$ is constructed in \cref{sec:stationary_Euler}, where we then analyze the associated linearized Euler operator and spectral projection. Finally, we complete the energy estimates by proving the bootstrap lemma in \cref{sec:Grenier}. Auxiliary lemmas are collected in the appendix.

\section{Outline}\label{sec:Preliminary}
Our analysis has two main components.
\begin{itemize}
     \item We perform spectral analysis for the linearized Euler equation at a non-shear steady state. In this part, we construct a smooth non-shear steady state $\uE$ near the sinusoidal flow, prove the static forcing $\nu \Delta \omE $ has non-trivial projection to the most unstable direction of $\cL$ and we can prove $\cO (1)$ separation in the time scale $\cO (\frac1{\Re \lambda} \log \frac1{\nu})$ if we drop the nonlinear term and the viscous term.    
     \item To generalize the linear separation result to the nonlinear separation result, we adapt a now-classical argument of Grenier \cite{grenier2000nonlinear} to decompose the nonlinear solution to \eqref{eq:w_cpt} as a finite sum of explicit ansatz together with a remainder. In contrast to the classical approach, we will not introduce an initial perturbation to the system, and the growth is purely driven by the Riesz projection of the force $\nu \Delta \omE$. By carefully analyzing the Grenier decomposition, the final separation estimate can be obtained. As a side comment, the viscous error $\nu \Delta w$ in \eqref{eq:w_cpt} can be controlled in a straightforward fashion by decomposing it into structured components and a controllable viscous remainder.  
\end{itemize} 

To start with, we consider special Euler steady states, which are given by semilinear elliptic equations. Consider the vorticity $ \omE$ whose associated stream function $\yE$ satisfies the equation
\begin{align}
    \label{eqn:semilinear}
    \Delta \yE = \omE = F (\yE)
\end{align}
for some $C^1$ function $F$ with nonvanishing derivative. It is standard to check that
\begin{align}\label{semi_rel_1}
    \nabla  \omE = \nabla  F (\yE) = F' (\yE) \nabla \yE \quad \leadsto \quad \uE \cdot \nabla \omE= \nabla ^\perp \yE \cdot F'(\yE) \nabla \yE \equiv 0,
\end{align}
and the steady Euler equation \eqref{eq:steady-Euler} is solved. 
 
\subsection{Spectral analysis of the linearized Euler equation near steady states}
Our argument will center around the Riesz projections $\P _\lambda$ of the linearized Euler operator $\cL$. Let $\lambda$ be an isolated point in the spectrum of $\cL$ defined in $\cX$, endowed with $L ^2$ inner product $(\cdot, \cdot)$. The Riesz projection is defined as
\begin{align}
    \label{eqn:riesz}
    \P _\lambda f := \frac1{2 \pi i} \oint _{\Gamma _\lambda} \cR _z (\cL)f \d z = \frac1{2 \pi i} \oint _{\Gamma _\lambda} (z - \cL) ^{-1}f \d z,
\end{align}
where $\Gamma _\lambda$ is a positively oriented simple closed contour around $\lambda$ and nothing else in the spectrum of $\cL$, and $\cR _z (\cL)$ is known as the resolvent. The following lemma marks the starting point of our discussion.
    
\begin{lemma}[Riesz projection]
    \label{lem:riesz}
    Suppose $\yE$ satisfies \eqref{eqn:semilinear} with $\omE:=\Delta \yE \in \cX$, where $F \in C ^1$ has a nonvanishing derivative. Then $\cL$ in \eqref{eq:w_cpt} is a densely defined closed operator on $\cX$ \eqref{def:X}. There exist a skew-adjoint operator $\cJ$ $(\cJ^*=-\cJ)$ and a bounded self-adjoint operator $\HX$ $(\HX^*=\HX)$ such that $\cL$ can be factorized as $\cL = \cJ \HX$ on the domain $D (\cL) \subset \cX$. If $\lambda \ne 0$ is an isolated eigenvalue of $\cL$ with algebraic multiplicity one and associated with an eigenvector $\phi$, then the following assertions hold:
    \begin{itemize}
    \item The $\bar \lambda$ is an algebraically simple eigenvalue of the adjoint operator $\cL ^*$, associated with an eigenfunction $\phi ^*$, and the Riesz projection \eqref{eqn:riesz} is equal to 
    \begin{align}
        \label{eqn:riesz-equiv}
        \P _\lambda f = \frac{(f, \phi ^*)}{(\phi, \phi ^*)} \phi.
    \end{align}
    
    \item 
    Further assume that the $-\bar \lambda$ is another algebraically simple eigenvalue of $\cL$. Let $\phi^-$ be the associated eigenfunction, then one can take $\phi^*=\HX\phi^-$ in \eqref{eqn:riesz-equiv}.
    \end{itemize}
\end{lemma}

The proof of \cref{lem:riesz} will be presented in Appendix \ref{sec:spectral_Euler}. Our first goal is to apply Lemma \ref{lem:riesz} to show that the Riesz projection of the forcing $\nu\Delta \omE$ \eqref{eq:w_cpt} onto the most unstable eigen-space of $\cL$ (with eigenvalue $\lambda$)  
is nontrivial. To this end, two conditions are essential: 
\begin{enumerate}[a)] 
    \item The eigenvalue $\lambda$ must be algebraically simple and isolated. 
    \label{item: simple}
    \item The inner product $(\Delta \omE, \HX \phi^-) = (\Delta \omE, \cH \phi^-)$ is nontrivial, where $\phi^-$ is a stable eigenfunction of $\cL$.
    \label{item: inner-prod}
\end{enumerate}

For the Euler steady states determined by the semilinear elliptic equation \eqref{eqn:semilinear}, the linearized operator $\cL$ has nice structures. The nonlocal operator $\cK$ in \eqref{eq:w_cpt} can be rephrased using \eqref{semi_rel_1},
\begin{align*}
    \cK g=-\nabla \omE \cdot \curl ^{-1} g = \nabla ^\perp \omE \cdot \nabla \Delta ^{-1}g = F' (\yE) \uE \cdot \nabla \Delta ^{-1}g.
\end{align*}
Since $F' (\yE)$ is invariant along the flow $\uE$, we observe that $\uE\cdot \nabla (F'(\yE))\equiv 0$, and  the operator $\cL$ in \eqref{eq:w_cpt} can be rephrased as a composition
\begin{equation}
\begin{aligned}
    \cL g & = -\uE \cdot \nabla g+ F' (\yE) \uE \cdot \nabla \Delta ^{-1} g
    = F' (\yE) \uE \cdot \nabla \left(
        -\frac g{F' (\yE)} + \Delta ^{-1} g 
    \right) = \cJ  \cH g=\cJ \HX g,\\
    \cJ &= F' (\yE) \uE \cdot \nabla, \qquad 
    \cH = -\frac1{F' (\yE)} + \Delta ^{-1},\qquad \HX = \P_{\ne0}\left[-\frac{1}{F'(\yE)}+\Delta^{-1}\right]\bigg|_{\cX}.
\end{aligned} \label{eqn:defn-J-H}
\end{equation}
Here $\cH: L ^2 (\T ^2) \to L ^2 (\T ^2)$ does not necessarily have mean-free range, so we apply the mean-free projection $\P _{\ne 0}$ to make $\HX: \cX \to \cX$ well-defined.

The establishment of conditions \ref{item: simple}) and \ref{item: inner-prod}) are intertwined and account for the most delicate part of our proof. To achieve conditions \ref{item: simple}) and \ref{item: inner-prod}), we restrict ourselves to the functional space that is symmetric about two points and invariant under a horizontal shift, given by \eqref{def:X}. The 
$\psi_0=\cos{5y}$ is an element in $\cX$, which is the stream function to a sinusoidal shear flow  \begin{equation}\label{Kolmogorov_flow}
    u _0 = \binom{5\sin 5y}{0}.
\end{equation}
We now define the associated linearized operator near $\omega_0:=-25\cos{5y}$ from $\cX$ to itself:
\begin{equation}
    \label{FLJH_0}
    \cL _0 =-125\sin{5y} \; \partial_x \circ \left(
        \frac{1}{25}+\Delta^{-1}
    \right) =: \cJ_0 \circ \cH_0.
\end{equation}
Note that $\cH = \HX$ for $u _0$.
Its spectral property is introduced in the next proposition, whose proof is given at the end of \cref{sec:kolmogorov}.

\begin{proposition}\label{pro:Kolmogorov}
    The linearized Euler operator $\cL_0 : \cX \supset D (\cL _0) \to \cX$ \eqref{FLJH_0} associated with the sinusoidal flow \eqref{Kolmogorov_flow}, with domain 
    \begin{align*}
        D(\cL_0) := \{f\in \cX \mid u_0 \cdot \nabla f\in \cX \text{ in the sense of distributions}\},
    \end{align*}
    has real algebraically simple isolated eigenvalues $\lambda_0>0$ and $ -\lambda_0<0$. Specifically, there exists $\phi_0\in D(\cL_0)$ such that
    \begin{align}
        \cL_0 \phi_0=\lambda_0 \phi_0,\quad  \lambda_0>0.\label{unstable_lambda_0}
    \end{align}
    Let $\sigma(\cL_0)$ denote the spectrum of the operator $\cL_0$, then $\sigma(\cL_0) \setminus \{-\lambda_0, \lambda_0\} \subset i\R$.
\end{proposition}

As we discussed in the previous section, $\Delta \omega_0$ would have trivial projection to the spaces spanned by $\phi_0$. This inspires us to find the steady state near the sinusoidal flow whose stream function lies in $\cX$.

\begin{proposition}\label{pro:Euler_Steady}
    There exists a universal constant $\e ^* > 0$, such that for $0 \le \e \le \e ^*$, there exists a Euler steady state $\ome \in \cX$, such that the associated stream function $\ye \in C ^\infty (\T ^2)$ solves the elliptic equation $\Delta \ye = \Fe (\ye)$, and $(\ye, \Fe)$ has asymptotic expansions
    \begin{align*}
        \Fe (s) &= -25 s + \e (2 s ^2 - 1) + \cO (\e ^2), \\
        \ye (x, y)
        & = 
        \cos 5 y + \e \left(
            \cos (3 x + 4 y)
            -\frac1{75} \cos 10 y 
        \right) + \cO (\e ^2).
    \end{align*}
    $\ye$ is analytic in $x$ and $y$, whereas $\Fe$ is a cubic polynomial in $s$. The remainder of $\Fe$ is an $s$-polynomial whose coefficients have size $\cO (\e ^2)$. The remainder of $\ye$ has size $\cO (\e ^2)$ in every fixed Sobolev norm: for all $k \in \N$, there is a $B _k>0$, such that $\|\ye-\psi_0\|_{H ^k (\T ^2)} \leq B _k \e$. 
\end{proposition}

The proof of this construction will be the goal of \cref{sec:fixed-point}. Next, we aim to study the mapping properties of the semigroup generated by the linearized Euler operator \eqref{eq:w_cpt} 
\begin{equation}\label{L_ep}
    \cL_\e w :=- \ue \cdot \nabla w - \nabla \ome \cdot \curl^{-1} w 
\end{equation}
in various Sobolev spaces. Let us pick a fixed large regularity level $m\in \N$ that satisfies
\begin{align}\label{choice_m}
    m \ge 2 + \frac{1000}{\lambda_0}. 
\end{align}
Since $\lambda_0$ is a universal constant, the level $m$ is universal. Under this setup, the spectral property of $\cLe$ is presented in the next proposition.

\begin{proposition}\label{pro:Euler_spectral}
For $0 \le k \le 2m$, define 
\begin{align}
    \label{eqn: Xk}
    \cX ^k = \cX \cap H ^k(\T ^2).
\end{align}
Let $\lambda_0$ be the unstable eigenvalue of $\cL _0$ \eqref{unstable_lambda_0}. There exists a threshold $\e ^{**} \le \e ^*$, such that for all $\e \in [0, \e ^{**}]$, the following statements hold.
\begin{enumerate}
    \item \label{item:Euler_spectral_1} Consider the linearized operator $\cLe: \cX ^k \supset D ^k (\cLe) \to \cX ^k$ with domain
    \begin{align*}
        D ^k (\cLe) := \{f \in \cX ^k \mid \ue \cdot \nabla f \in \cX ^k \text{ in the sense of distributions}\}.
    \end{align*}
    The most unstable eigenvalue $\lambda _\e$ is positive and simple, with an eigenfunction $\fe \in \cX ^k$. Moreover, $\lambda _\e \to \lambda _0$ as $\e \to 0$. 

    \item \label{item:nontrivial-proj}
    Let $\P _{\lambda_\e}$ be the Riesz
    projection of $\cLe$ associated with $\lambda _\e$. There exists a constant $\kappa>0$ such that
    \begin{align*}
    \|\P_{\lambda_\e} \Delta\ome\|_{L^2}
    \geq
    \kappa\e^2.
    \end{align*}    
    In particular, $\P_{\lambda_\e} \Delta\ome \ne 0$ for all sufficiently small positive $\e \le \e ^{**}$.

    \item \label{item:Euler_spectral_3} 
    Define the complementary Riesz projection 
    \begin{equation}\label{defn:Q,P}
     \Q_{\lambda_\e} := \I - \P_{\lambda_\e}.
    \end{equation} 
    For $t \ge 0$, we have the relation  
    $\expo{t\cLe}\P_{\lambda_\e}
    =\expo{\lambda_\e t}\P_{\lambda_\e}$, and the bound 
    \begin{equation}
    \label{eq:euler-complement-outline}
        \|\expo{t\cLe}\Q _{\lambda_\e}\| _{\cX ^k \to \cX ^k}
        \le M \expo{\lambda _\e t / 2}, \qquad t \ge 0
    \end{equation}
    for some constant $M$ depending on $\ue$ and $m$.
\end{enumerate}
\end{proposition}

We prove Item \ref{item:nontrivial-proj}~in \cref{sec:riesz}, and Item \ref{item:Euler_spectral_1} and \ref{item:Euler_spectral_3} follow standard arguments (see \cites{kato1995perturbation,ShvydkoyLatushkin2005,ShvydkoyFriedlander2008,lin2013unstable}) and are postponed in Appendix \ref{sec:spectral_Euler}.

\subsection{Energetic analysis}
In this section, we design the Grenier iteration scheme. First of all, we determine the level of approximation. Let us consider the parameter regime $\e \le \e ^{**}$ defined in Proposition \ref{pro:Euler_spectral}, and recall the Euler steady state $\omE =\ome$ with the unstable  eigenvalue $\lambda=\lambda_\e>0$ from Proposition \ref{pro:Euler_spectral}.

Let us consider the following decomposition of the Navier--Stokes solution \eqref{eq:NS_vor} at level $m$: \begin{equation}
    \omnu:=\sum_{k=0}^{m-1}\nu^{k}\omega _k+\tilde \omega,\quad u_k:=\operatorname{curl}^{-1}\omega_k,\quad \tilde u:=\operatorname{curl}^{-1}\tilde \omega,
\end{equation} 
where $\omega _0 = \omE$ and the $\omega _k$ solves the equation
\begin{align}
   \partial _t \omega _k + \sum _{\substack{i, j \ge 0 \\i + j = k}} u _i \cdot \nabla \omega _j  &= \Delta \omega _{k - 1},\quad\omega_k(t=0)\equiv0, \quad\text{ for } 1 \le k \le m - 1.
\end{align}
These equations can be iteratively solved by the Duhamel's Formula: $\om _k (t) = \int _0 ^t \expo{(t - s) \cL}  f _k (s) \d s$ for some $f _k$ depending on $\omega _j$, $j < k$.
The remainder $\tilde \omega$ initiates from $\tilde \omega(t=0)\equiv0$, and satisfies the equation
\begin{equation}
    \label{eqn:remainder}
    \partial _t \tilde \omega + \unu \cdot \nabla \tilde \omega + \tilde u \cdot \nabla \left(\sum _{k = 0} ^{m - 1} \nu ^k \omega _k\right) = \nu ^m \Delta \omega _{m - 1} + \nu \Delta \tilde \omega - \sum _{\substack{1 \le i, j \le m - 1 \\i + j \ge m}} \nu ^{i + j} u _i \cdot \nabla \omega _j.
\end{equation}
We need the following proposition. 
\begin{proposition}\label{pro:reg}
    Fix approximation level $m$ as in \eqref{choice_m}. For $\e$ sufficiently small, we have $(m - 2) \lambda _\e \ge 4 \| \nabla \ome \| _{L ^\infty}$. Now we set $\omE=\ome$, $\lambda=\lambda_\e$. For $1 \le k \le m - 1$, there exists constant $C_k=C_k(\omE, \lambda, m)$ such that 
    \begin{align}     
        \| \omega _k (t) \| _{H ^{2 (m + 1 - k)}} &\le C _k \expo{k \lambda t}, \qquad \forall t \ge 0.
    \end{align}
    There exists a constant $\tilde C = \tilde C (\omE, \lambda, m)$, such that for each $\delta \le \delta ^{\mathrm E}$ and $\nu \le c ^{\mathrm E} \delta$, the following estimates are satisfied:
    \begin{align}
        \| \tilde{\omega} (t) \| _{L ^2} &\le \tilde C \sum _{j = m - 1} ^{2 m} \nu ^j \expo{j \lambda t}, \qquad \text{ for } t \le \frac1\lambda\log\left(\frac\delta\nu\right).
    \end{align}
\end{proposition}
The proposition will be proven in \cref{sec:Grenier}. Once the proposition is developed, we can conclude the proof of Theorem \ref{thm:main}.

\begin{proof}[Proof of Theorem \ref{thm:main}]First of all, we fix the $\e>0$ such that all the statements in \cref{pro:Euler_Steady,pro:Euler_spectral,pro:reg} hold and declare that all the constants $C_1,C_2,\cdots$ in this proof depends on the choice of $\omE := \ome$. The only notation that we will bring from Proposition \ref{pro:Euler_spectral} is the eigenvalue $\lambda$. Next, we observe that the first approximation $\omega_1$ solves the equation ($\cL:=\cL_\e$)
\begin{equation}
 \partial _t \omega _1 = \cL\omega_1+\Delta \omE,\quad \omega_1(t=0)\equiv 0. 
 \end{equation} From Proposition \ref{pro:Euler_Steady}, we expect the $\omega_1$-component undergoes exponential growth with rate $\lambda>0$. Indeed, one can rewrite the $\omega_1$-component in Duhamel form using \eqref{defn:Q,P}
\begin{equation}
    \omega _1 (t) = \int _0 ^t \expo{(t - s) \cL} \Delta \omE \d s = \frac{\expo{\lambda t} - 1}{\lambda} \P_\lambda \Delta \omE + \int _0 ^t \expo{(t - s) \cL} \Q_\lambda \Delta \omE \d s.
\end{equation}
The complementary Riesz
$\Q _\lambda$-bound \eqref{eq:euler-complement-outline} yields $\left\| \int _0 ^t \expo{(t - s) \cL} \Q_\lambda \Delta \omE \d s \right\| _{L ^2} \le C_1 \expo{\lambda t / 2}.$ 
Since $\lambda\in \R$, and the Riesz projection $\P_\lambda \Delta \omE$ is nonzero  (Item \ref{item:nontrivial-proj},  Proposition \ref{pro:Euler_spectral}), we have that
\begin{align}\label{pf1.2_1}
    \| \omega _1 (t) \| _{L ^2} \ge \frac1{C_2} (\expo{\lambda t}-1)- C_1 \expo{\lambda t / 2} \ge \frac{1}{2C_2} \expo{\lambda t}, \qquad \forall t \geq \frac2{\lambda} \log(2+2C_1C_2).
\end{align}

The remaining task is to show that all the remaining components have small contribution. Thanks to Proposition \ref{pro:reg}, for all $\delta\leq1/2$ and for any $t$ satisfying $\nu \expo{\lambda t} \le \delta$,
\begin{align}
    \label{pf1.2_2}
    \left\| \sum _{k = 2} ^{m - 1} \nu ^k \omega _k + \tilde \omega \right\| _{L ^2} \le C _3 \sum _{k = 2} ^{2 m} \nu ^k \expo{k \lambda t} \le 2 C _3 \delta ^2.
\end{align}
Recalling the parameter regime $\{\cE \le (2+2C_1C_2)^{-2}$, $0<\nu\leq \cE\delta$\}, we set $t _\nu ^\delta:= \frac1{\lambda} \log \frac\delta\nu\geq\frac{2}{\lambda}\log(2+2C_1C_2)$. Hence, the relations \eqref{pf1.2_1} and \eqref{pf1.2_2} imply the vorticity separation
\begin{align*}
    \| \nu \omega _1 (t_\nu ^\delta) \| _{L ^2} \ge \frac \nu {2C_2} \expo{\lambda t_\nu ^\delta} = \frac\delta {2C_2}\quad \leadsto \quad  \| \omega (t_\nu ^\delta)- \omE \| _{L ^2} \ge \frac\delta{2C_2}   - 2{C_3} \delta ^2 \ge \frac\delta{4C_2},
\end{align*}
as long as $\dE \leq \frac{1}{8C_2C_3}$ and $0<\delta\le \dE$.

Finally, we consider the velocity difference $\uE-\unu$. The Biot--Savart law yields that 
\begin{align*}
    u _1=\curl^{-1}\om_1 = \frac{\expo{\lambda t} - 1}{\lambda} \curl ^{-1} \P_\lambda \Delta \omE + \int _0 ^t \curl ^{-1} \expo{(t - s) \cL} \Q_\lambda \Delta \omE \d s.
\end{align*}
By a similar argument as in \eqref{pf1.2_1}--\eqref{pf1.2_2}, there exist a constant $C_4$ such that 
\begin{align}
   \quad  \| u(t _\nu ^\delta) - \uE \| _{L ^2} \ge \frac\delta{C_4}.
\end{align}
The remaining proof is straightforward and hence omitted.
\end{proof}
The remainder of the paper proceeds from observation to perturbation, and from perturbation to iteration. We then develop the spectral theory of the linearized Euler operator $\cL$ in two steps. In Section~\ref{sec:kolmogorov}, we begin with the unstable sinusoidal flow, following the classical approach. In Section~\ref{sec:stationary_Euler}, we build on this analysis and perturb the flow to construct an unstable Euler steady state for which $\P_\lambda\Delta\omE\ne0$. Finally, in Section~\ref{sec:Grenier}, we establish the regularity estimates for the Grenier iteration scheme.

\section{Spectral analysis of an unstable Kolmogorov flow}
\label{sec:kolmogorov}

\subsection{Existence of unstable eigenpair in the symmetry class \texorpdfstring{$\cX$}{X}}\label{sec:Unstable_e_function}

We apply the continued fraction method to find the eigenfunction. Technical Lemma \ref{prop:friedlander1998} from \cite{friedlander1998instability} helps us in the analysis. To prove Proposition \ref{pro:Kolmogorov}, we first identify an eigenvalue in this subsection, and show it is the unique unstable eigenvalue in the next subsection.

\begin{lemma}\label{lem_1:Kolmogorov}
    Under the assumption of Proposition \ref{pro:Kolmogorov}, the linearized Euler operator $\cL _0$ has a pair of real eigenvalues $\lambda_0>0$ and $-\lambda_0$. The associated smooth eigenfunctions $\phi_0$ and $\phi_0^-$ satisfy the relation 
    \begin{align}\label{Kol_equality}
        ( \phi _0, \cos (3 x - y) ) = -( \phi _0 ^-, \cos (3 x - y) ) \ne 0.
    \end{align}
\end{lemma}

\begin{proof}

Our proof is constructive. We create the eigenfunction as a cosine series, and then determine the coefficient by the eigen relation.

We focus on the unstable eigenpair $(\lambda_0,\phi_0)$, whose eigen-relation \eqref{FLJH_0} reads as follows
\begin{equation}
    \label{lem1:Kol_pf_1}
    \cL _0 \phi _0 = -5 \sin{5y} \cdot \partial_{x} (\I + 25 \Delta ^{-1}) \phi _0 =\lambda _0 \phi _0.
\end{equation}
To solve this equation,  we consider an ansatz $\phi _0\in \cX$ \eqref{eqn: fourier} whose Fourier modes are nonzero only if $j = 1$ and $5\vert(1+n)$, i.e., 
\begin{align}
    \label{eqn:ansatz-power}
    \phi _0 = \sum _{\substack{n \in \Z \\ n\equiv 4 (\mathrm{mod}\ 5)}} c _n \cos (3 x + n y).
\end{align}
Throughout the proof, we always work on the subset $\J :=\{n \in \Z \mid n \equiv 4 (\mathrm{mod} \ 5)\}$. One will uncover further constraints on the Fourier coefficients as the analysis continue. It is straightforward to check that from our ansatz \eqref{eqn:ansatz-power} we get 
\begin{align}\label{lem1:Kol_pf_1.5}
    (\I + 25 \Delta ^{-1}) \phi _0 = \sum _{n \in \J} h _n c _n \cos (3 x + n y),\quad  h _n := \frac{n ^2 - 16}{n ^2 + 9}.
\end{align}
Thanks to the relation 
$ \displaystyle
    -5 \sin 5 y \; \partial _x \cos (3x + n y) = \frac{15}2 \cos (3 x + (n - 5) y) - \frac{15}2 \cos (3 x + (n + 5) y),$ 
we obtain 
\begin{align*}
    \cL _0 \phi _0 &= \frac{15}2 \sum _{n \in \J} (h _{n + 5} c _{n + 5} - h _{n - 5} c _{n - 5})  \cos (3 x + n y).
\end{align*}
Hence, the eigen-relation \eqref{lem1:Kol_pf_1} yields the inductive relationship
\begin{align}\label{lem1:Kol_pf2}
    \frac{15}2 (h _{n + 5} c _{n + 5} - h _{n - 5} c _{n - 5}) = \lambda _0 c _n, \qquad n \equiv 4 ({\rm mod }\ 5).
\end{align}
We note three relations: $h _4 = 0$, $h _{-1} =-\frac32$, and $h _n > 0$ for all the other $n \in \J$. The vanishing of $h_4$ implies that if we choose $c_9=0$, the relation $c_9 = c _{14} = c _{19} = \dots = 0$ holds. We take this ansatz from now on. For $n = 4$, the above relationship reduces to $-\frac{15}2 h _{-1} c _{-1} = \lambda _0 c _4$.

Next, we use dontinued fraction to determine the unstable eigenvalue. For $n \le -1$, we set $\beta _n= \frac{2}{15 h _n}$ and observe three relations $\beta_{-1}=-\frac{4}{45}$, $\lim _{n \to -\infty} \beta _n = \frac2{15}$ and $\beta _n > 0$ for all $n < -1$. Further consider the ratio $r _n =\frac{h_{n+5} c_{n+5}}{h_n c_n}$. The condition $h_4=0$ and the relation \eqref{lem1:Kol_pf2} yield that
\begin{align}\label{lem1:Kol_pf3}
   r _{-1} = 0;\quad r _n = \lambda _0 \beta _n + \frac1{r _{n - 5}},\quad n\equiv 4({\rm mod}\ 5),\quad n\le-1.
\end{align}
Solving this induction relation naturally leads to a continued fraction expression
\begin{align}
    r _{n} = \lambda _0 \beta _{n} +\cF_n(\lambda_0),
    \quad  \cF_n(\lambda_0)
    &:=[\lambda_0\beta_{n-5},\lambda_0\beta_{n-10},\lambda_0\beta_{n-15},\cdots]\\
    &:=\dfrac1{\lambda _0 \beta _{n-5} + \dfrac1{\lambda _0 \beta _{n-10} + \dfrac1{\lambda _0 \beta _{n-15} + \cdots}}}.\label{lem1:Kol_pf4}
\end{align}
Combining the relation $r_{-1}=0$ and the expression above, we end up with the equation $0=r _{-1} = \lambda _0 \beta _{-1} +\cF_{-1}(\lambda_0).$ This relation helps us determine the unstable eigenvalue $\lambda_0>0$. By Lemma \ref{prop:friedlander1998}, we realize that the continuous fraction expression $\cF _{-1}(\lambda) := [\lambda \beta_{-6}, \lambda \beta_{-11}, \cdots]$ is convergent and continuous for all $\lambda > 0$, with
$\lim _{\lambda \to 0 ^+}\cF_{-1}(\lambda) = 1,\,  \lim _{\lambda \to +\infty} \cF_{-1}(\lambda) = 0.$
By intermediate value theorem applied to the function $\cF_{-1}(\lambda)+\beta_{-1}\lambda$, there exists $\lambda = \lambda _0>0$ such that 
$\cF_{-1}(\lambda _0) = -\lambda _0 \beta _{-1} = \frac4{45} \lambda _0.$ This positive parameter $\lambda_0>0$ is our unstable eigenvalue.

Finally, we determine the Fourier coefficients. We observe that since $h_n$ is given by \eqref{lem1:Kol_pf_1.5} and the ratio $r_n$ is given by \eqref{lem1:Kol_pf4}, the numerical value of $c_4$ would uniquely determine all the nontrivial Fourier coefficients $\{c_n\}_{n\in \J\cap\Z_-}$. 
Thanks to the expression \eqref{lem1:Kol_pf4} and Lemma \ref{prop:friedlander1998}, for all $ n\in \J \cap\Z_-$, the ratio $r _n = \lambda _0 \beta _n + \cF _n(\lambda_0)$ is well-defined. Moreover, the limit \eqref{fried_con1} and Remark \ref{rmk:Fried_con} yield that 
\begin{align*}
    \lim _{n \to -\infty} \cF _n (\lambda_0) = \sqrt{\left(\frac{\lambda _0}{15} \right) ^2 + 1} - \frac{\lambda _0}{15}.
\end{align*}
Thus
\begin{align*}
  r _\infty :=   \lim _{n \to -\infty} r _n = \sqrt{\left(\frac{\lambda _0}{15} \right) ^2 + 1} + \frac{\lambda _0}{15} > 1,
\end{align*}
and consequently 
\begin{align*}
    \lim _{n \to -\infty} \frac{c _{n - 5}}{c _n} = \lim _{n \to -\infty} \frac{h _n}{h _{n - 5}} \cdot \frac1{r _{n - 5}} = \frac1{r _\infty} < 1.
\end{align*}
Since Fourier coefficients converge to zero geometrically, the eigenfunction 
\begin{align*}
    \phi _0 (x, y) = c_4\cos(3x+4y)+ \sum _{n \in \J \cap \Z _-} c _n \cos (3 x + n y)
\end{align*}
is analytic. Moreover, all the Fourier coefficients $c _n,\ n\in\J\cap\Z_- $ are nonzero.

One can directly verify
\begin{align*}
    \phi _0 ^- (x, y) =c_4\cos(3x+4y)+ \sum _{n \in \J \cap \Z _-} (-1) ^n c _n \cos (3 x + n y) = \phi _0 (x, y + \pi)
\end{align*}
is a stable eigenfunction of $\cL _0$, with eigenvalue $-\lambda _0$. The relation \eqref{Kol_equality} follows the fact that $c _{-1} \ne 0$. 
\end{proof}

\subsection{Proof of Proposition \ref{pro:Kolmogorov}}

In this section, we conclude the proof of Proposition \ref{pro:Kolmogorov} by showing that the eigenvalue $\pm \lambda_0$ that we identified in the last section is algebraically simple and isolated.

\begin{proof}[Proof of Proposition \ref{pro:Kolmogorov}]
    In Lemma \ref{lem_1:Kolmogorov} we have obtained a pair of eigenvalues $\pm \lambda _0$ with associated eigenfunctions $\phi _0$ and $\phi _0 ^-$.
    The uniqueness and simplicity of the unstable eigenvalue is guaranteed by the index formula from \cite{LinZeng2022}*{theorem 2.3}, which we cited as Lemma \ref{lem: lin-zeng} in \cref{sec:Hamiltonian}. Note that the invariant subspace of $\cH _0 = \frac1{25} + \Delta ^{-1}$ are the Fourier modes. Fourier modes in $\cX$ are given by \eqref{eqn: fourier}, and 
    \begin{align*}
        \cH _0 \cos (3 j x + n y) = 
        \left( \frac1{25} - \frac1{(3 j) ^2 + n ^2} \right) \cos (3 j x + n y).
    \end{align*}
    Since \eqref{eqn: fourier} requires $5 \mid j + n$, we know $\cH _0$ has only one negative direction which spans $\cX _- = \operatorname{Span} \{\cos (3x - y)\}$, and has a two dimensional kernel $\ker \cH _0 = \operatorname{Span} \{\cos 5 y, \cos (3x + 4y) \}$. All the other directions in $\cX _+$ are positive, and 
    \begin{align*}
        \cX = \cX _- \oplus \ker \cH _0 \oplus \cX _+.
    \end{align*}
    
    Let us verify the assumptions of Lemma \ref{lem: lin-zeng}. In our setting, $\mf X = \cX$ is reflexive, and the dual space is equivalent to $\cX$ itself, hence the duality pairing is the same as the inner product, and the dual operator is the same as the adjoint operator. In \cref{sec:functional} we verified (H1) and (H2). For (H3), since our $\cX _-$, $\ker \cH _0$, $\cX _+$ are mutually orthogonal, we conclude 
    \begin{align*}
        \left\{
            f \in \cX ^*: \langle f, u \rangle = 0, \forall u \in \cX _- \oplus \cX _+ 
        \right\} = \ker \cH _0 = \operatorname{Span} \{\cos 5 y, \cos (3x + 4y) \},
    \end{align*}
    which are in the domain of $\cJ$. Indeed, apart from the imposed symmetry, we are in the same setting as \cite{LinZeng2022}*{section 11.5}.
    
    We use the conclusion (iii) of \cref{lem: lin-zeng}. $\cH _0$ has exactly one negative direction, thus $n ^- (\cH _0) = 1$. Since we have already found one eigenfunction with a real eigenvalue $\lambda _0 > 0$, we know $k _r \ge 1$, and the index formula enforces $k _r = 1$ and $k _c = 0$, which means there are no other eigenvalue with positive real part, and $\lambda _0$ is simple. By (ii), $-\lambda _0$ is the only eigenvalue with negative real part. Moreover, because $\cL_0 = \cT _0 + \cK_0$ where $\cT _0 = \cJ _0 / 25$ and $\cK _0 = \cJ _0 \Delta^{-1}$ is compact, the essential spectrum of $\cL _0$ coincide with the essential spectrum of a transport operator $\cT _0$, and $\sigma (\cT _0) \subset i \R$ because $\cT _0$ is skew adjoint. Therefore $\sigma (\cL_0) \setminus i \R = \{\pm \lambda _0\}$.
\end{proof}

\section{Construction of the stationary Euler solution}\label{sec:stationary_Euler}

In this section, we prove \cref{pro:Euler_Steady,pro:Euler_spectral}. Construction of the steady state $\ue$ is proven in \cref{sec:fixed-point} using Banach fixed point theorem, and nontrivial Riesz projection is proven in \cref{sec:riesz}. 

\subsection{Perturbation of the sinusoidal flow}
\label{sec:fixed-point}

We provide a similar steady state to the work of \cite{coti2023stationary}. They perturbed the Kolmogorov flow $\psi _0 = \cos y$ to $\ye$, which satisfies a semilinear equation with quintic nonlinearity. In our construction, we used a cubic polynomial instead for simplicity. The tradeoff is that our stream function $\ye$ is not mean-zero. Nevertheless, the stream function still preserves the symmetry of the function class $\cX$ defined in \eqref{def:X}. For this reason, let us introduce the $\cY$ space, which is the preimage of $\cX$ under the Laplacian operator:
\begin{align}
    \label{def: Y2}
    \cY &= \left\{ 
        f \in H ^2 (\T ^2): f (x, y) = f (-x, -y) = f \left(x + \frac{2 \pi} 3, y\right) = f \left( \frac{2 \pi}{15} - x, \frac{2 \pi}{5} - y \right)
    \right\},
\end{align}
equipped with the weighted $H ^2$ norm 
\begin{align*}
    \| f \|_{\cY} := \| f \| _{L ^2 (\T ^2)} + \frac{1}{225} \| f \| _{\dot H ^2 (\T ^2)}.
\end{align*}
Now, we denote  
\begin{align}
\label{Ball_psi}
\cB _\psi &:= \left\{
    \psi \in \cY: \left\| 
        \psi - \cos 5y - \e \cos (3x + 4y) + \frac\e{75} \cos 10 y
    \right\|_{\cY} \le 2 \e ^2
\right\}, 
\end{align}
be the closed $2 \e ^2$ ball, $\e$ away from $\cos{5y}$, which is the stream function of the sinusoidal flow. We are going to construct non-shear steady state of 2D Euler in $\cB _\psi$.

\begin{lemma}
    \label{lem:fixed-point}
    Let $T _1$, $T _2$, $T _3$ be the Chebyshev polynomials:
    \begin{align}
        T _1 (s) &= s, &
        T_2 (s) &= 2 s ^2 - 1, &
        T_3 (s) &= 4 s ^3 - 3 s.
    \end{align}
    For every $\e > 0$ sufficiently small, there exist constants $\Ae, \Be \in [-\e, \e]$ and $\ye \in \cB _\psi$  
    such that $\Delta \ye = \Fe (\ye)$, with $\Fe$ given by a cubic polynomial 
    \begin{align}
        \label{eq:Fe-defn}
        \Fe (s) = -25 s + \e \left(
            \Ae s + \Be T_3 (s) + T_2 (s)
        \right).
    \end{align}
\end{lemma}

\begin{proof}
Let $\{e _1, e _2\} := \{\cos 5y, \cos (3x + 4y)\}$ be the orthogonal basis of $\ker (25 + \Delta) \subset \cY$. Throughout the proof, $(25+\Delta)^{-1}$ denotes the bounded inverse from $\ker(25+\Delta)^\perp \subset \cY$ to $\cY$, and $\P_{\ker^\perp}$ denotes the corresponding $L^2$ orthogonal projection. Our strategy is to identify solutions to the semilinear elliptic equation $\Delta\ye=\Fe(\ye)$ of the form  
\begin{align}\label{defn:varphi}
    \ye &= e _1 + \e e _2 + \e \vfe, \quad  (\vfe,e_i) =0, \qquad i=1,2.
\end{align}
Here $(\cdot,\cdot)$ refers to the $L ^2$ inner product. We refer to $\vfe$ as the \emph{first-order correction}. Then, finding $\ye \in \cB _\psi$ is equivalent to finding $\vfe \in \cB _\varrho$, defined as
\begin{align}
    \label{Ball_varphi}
    \cB _\varrho &:= \left\{ 
        \varrho \in \ker (25 + \Delta) ^\perp: \left\| \varrho + \frac1{75} \cos 10 y \right\|_{\cY} \le 2 \e
    \right\}.
\end{align}
Combining this ansatz in \eqref{defn:varphi} and the targeted structure of the nonlinearity $\Fe$ \eqref{eq:Fe-defn}, we observe that solving $\Delta \ye = \Fe (\ye)$ is equivalent to solving \eqref{defn:varphi} and the following:
\begin{align}
    \label{eqn: phi}
    (25 + \Delta) \vfe = \Ae \ye + \Be T_3 (\ye) + T_2 (\ye),
\end{align}
In the sequel, we will drop the $(\cdots)_\e$ for the sake of cleanness of notation. 

For $\varrho \in \cB _\varrho$, denote
\begin{align*}
    \psi _\varrho := e _1 + \e e _2 + \e \varrho,
\end{align*}
then expanding Chebyshev polynomials yields
\begin{align}
    \label{def:T_psi}
    A \psi _\varrho + B T _3 (\psi _\varrho) + T _2 (\psi _\varrho) &= A \cos 5y + B \cos 15 y + \cos 10 y \\
    & \qquad + \e (A + 3B + 6 B \cos 10y + 4 e _1) \cdot (e _2 + \varrho) + \e ^2 P, \\
    \text{where } P = P (\e, e _1, e _2, B, \varrho) & :=(12Be_1+2)(e_2+\varrho)^2+4B\e(e_2+\varrho)^3.
\end{align}
As $(25 + \Delta)$ is self-adjoint, solving \eqref{eqn: phi} is equivalent to solve $(A, B, \varrho)$ satisfying
\begin{align} \label{eqn: orthogonality}
    ( A \psi _\varrho + B T _3 (\psi _\varrho) + T _2 (\psi _\varrho),e_i ) _{L ^2} = 0, \quad i=1,2,
\end{align}
and 
\begin{align}\label{eqn:pseduinverse}
    \varrho = (25 + \Delta) ^{-1} (A \psi _\varrho + B T_3 (\psi _\varrho) + T_2 (\psi _\varrho)).
\end{align}
Plugging in \eqref{def:T_psi}, by orthogonality, \eqref{eqn: orthogonality} expands to 
\begin{equation}
\begin{aligned}
   & 2 \pi ^2 A + \e \cdot 6 B ( e_1 \cos 10 y, \varrho ) + 4 \e ( e_1^2, \varrho ) + \e ^2 ( P,  e_1 ) = 0,\\
   & 2 \pi ^2 (A + 3 B) + 6 B ( e_2 \cos 10 y, \varrho ) + 4 ( e_2 e_1, \varrho ) + \e ( P,e_2 ) = 0.
\end{aligned}
\end{equation}
which is equivalent to 
\begin{subequations}
\begin{align}
    \label{eqn: A'}
    A &= -6 B \e \cdot \frac1{2 \pi ^2} ( e_1 \cos 10 y, \varrho ) - 4 \e \cdot \frac1{2 \pi ^2} ( e_1^2, \varrho ) - \e ^2 \cdot \frac1{2 \pi ^2} ( P, e_1 )=:\cS _A, \\
    \label{eqn: B'}
    B &= - \frac{A + 4 \cdot \frac1{2 \pi ^2} (  e_1e_2, \varrho ) + \e \cdot \frac1{2 \pi ^2} ( P, e_2 )}{3 + \frac6{2 \pi ^2} ( e_2 \cos 10 y, \varrho )}=:\cS _B.
\end{align}
\end{subequations}
We claim the following Lemma.  

\begin{lemma}
    \label{lemma:construction 1}
    If $\e$ is sufficiently small, then for every $\varrho \in \cB _\varrho$,
    \begin{align*}
        \cS: (A, B) \mapsto (\cS _A, \cS _B ) 
    \end{align*}
    is a contraction in $[- \e, \e] ^2$, measured in $\ell ^\infty$ norm.
\end{lemma}

Based on the above Lemma, Banach fixed-point theorem implies for $\e$ sufficiently small and $\varrho \in \cB _\varrho$, $\cS$ has a unique fixed point in $[- \e, \e] ^2$, and we define that to be $(A _\varrho, B _\varrho)$. Now \eqref{eqn:pseduinverse} is equivalent to 
\begin{equation}
   \varrho= \cT  (\varrho) := (25 + \Delta) ^{-1} (A _\varrho \psi _\varrho + B _\varrho T_3 (\psi _\varrho) + T_2 (\psi _\varrho)).
\end{equation}
To finish the proof of Lemma \ref{lem:fixed-point}, it suffices to prove $\cT $ is a contraction in $\cB _\varrho$ and it requires to build the Lipshitz bound for the mapping $\varrho \mapsto (A _\varrho, B _\varrho)$.
\begin{lemma}\label{lemma:construction 2}
    $\Lip (\varrho \mapsto (A _\varrho, B _\varrho)) \le 1$.
\end{lemma}
We postpone the proof of Lemma \ref{lemma:construction 1} and Lemma \ref{lemma:construction 2} for now and  show  the mapping $\cT $ is a contraction in $\cB _\varrho$.
 To verify this map is well-defined, we only need to show 
\begin{align}
    \label{eqn: boundedness-Tphi}
    \left\|\cT  ( \varrho) + \frac1{75} \cos 10 y \right\|_{\cY} \le 2 \e,
\end{align}
because $\cT  ( \varrho)$ is clearly orthogonal to $\ker (25 + \Delta)$. This can be shown by plugging in \eqref{def:T_psi}:
\begin{align}
    \notag
    \cT  ( \varrho)
    &= (25 + \Delta) ^{-1} \P _{\ker ^\perp} (A _\varrho \psi _\varrho + B _\varrho T_3 (\psi _\varrho) + T_2 (\psi _\varrho)) \\
    \label{eqn: PKerP}
    &= (25 + \Delta) ^{-1} (B _\varrho \cos 15y + \cos 10 y) \\
    \notag
    & \qquad + (25 + \Delta) ^{-1} \P _{\ker ^\perp} \left[ 
    \e \left( 
        A _\varrho + 3 B _\varrho + 6 B _\varrho \cos 10 y + 4 e _1
    \right) (e _2 + \varrho) + \cO (\e ^2)
    \right] \\
    \notag
    &= -\frac{B _\varrho}{200} \cos 15y - \frac1{75} \cos 10y \\
    & \qquad + \e \cdot (25 + \Delta) ^{-1} \P _{\ker ^\perp} \left[ 
    4 e _1 \cdot \left(e _2 - \frac1{75} \cos 10y\right)
    \right] + \cO (\e ^2).
    \label{eqn: Tphi-rhs}
\end{align}
Here the remainder term $\cO (\e ^2)$ is measured in $\cY$ norm. Note that $P (\e, e _1, e _2,  B _\varrho, \varrho)$, a polynomial of an $H ^2$ function, is still an $H ^2$ function, by the two-dimensional product estimate $\|fg\|_{H^2}\le C\|f\|_{H^2}\|g\|_{H^2}$. Moreover, 
\begin{align}
    \label{eqn: cos15}
    \left\| 
        \cos 15y
    \right\|_{\cY} = \left\| \cos 15 y \right\| _{L ^2} + \frac{1}{225} \left\| \cos 15 y \right\| _{\dot H ^2}
    = 2 \sqrt 2 \pi,
\end{align}
and 
\begin{align*}
    & \P _{\ker ^\perp} \left[ 
    4 e _1 \cdot \left(e _2 - \frac1{75} \cos 10y\right)
    \right] = 2\cos(3x+9y) + 2\cos(3x-y) - \frac2{75}\cos15y \\
    \leadsto \qquad 
    & \left\| (25 + \Delta) ^{-1} \P _{\ker ^\perp} \left[ 
    4 e _1 \cdot \left(e _2 - \frac1{75} \cos 10y\right)
    \right] \right\| _{\cY} \\
    & \qquad \le \frac2{65} \lVert \cos(3x+9y) \rVert _\cY + \frac2{15} \lVert \cos(3x-y) \rVert _\cY + \frac1{200} \cdot \frac2{75} \left\| \cos15y \right\| _\cY \\
    & \qquad \le \left(
        \frac2{65} + \frac2{15} + \frac1{7500}
    \right) 2 \sqrt 2 \pi \le \frac{\sqrt 2 \pi}3.
\end{align*}
Therefore, \eqref{eqn: Tphi-rhs} can be bounded by 
\begin{align}
\left\|\cT (\varrho) + \frac1{75}\cos10y \right\|_{\cY}
\le \frac{2\sqrt2\pi}{200} \e + \frac{\sqrt 2 \pi}3 \e + \cO (\e ^2) \le 2 \e 
\end{align}
for sufficiently small $\e$. Thus, $\cT$ is well-defined. Similarly, by \eqref{eqn: PKerP}, we have 
\begin{align*}
    \|\cT (\varrho_1) -\cT  (\varrho_2)\|_{\cY} &\le C \e |A _{\varrho _1} - A _{\varrho _2}| + \left( \frac{\sqrt 2 \pi}{100} + C \e\right) |B _{\varrho _1} - B _{\varrho _2}| + C \e \|\varrho _1 - \varrho _2\| _{\cY} \\
    & \le \frac{1}{2} \|\varrho_1 - \varrho_2\| _{\cY},
\end{align*}
applying \cref{lemma:construction 2}. By applying Banach fixed-point theorem to $\cT$ in $\cB _\varrho$, there exists a fixed point $\vfe$ of $\cT$ in $\cB _\varrho$. Therefore, $\ye := \cos{5y} + \e \cos{(3x+4y)} + \e \vfe$ is a steady state solving $\Delta \ye = \Fe (\ye)$ with $ \Fe (s) = -25 s + \e (
\Ae s + \Be T_3 (s) + T_2 (s))$, where $\Ae = A _{\vfe} \in [-\e, \e]$ and $\Be = B _{\vfe} \in [-\e, \e]$.
\end{proof}

Now we prove \cref{lemma:construction 1} and \cref{lemma:construction 2}.

\begin{proof}[Proof of \cref{lemma:construction 1}]
    We first verify that $|\cS _A| \le \e$, and the Lipschitz constant is $\cO (\e)$ as long as $\e$ is sufficiently small. This is because by Cauchy--Schwartz, we have 
    \begin{align*}
        \left|(e _1 ^2, \varrho)\right| &\le \frac1{75} \left|(e _1 ^2, \cos 10 y)\right| + \left|\left( e _1 ^2, \varrho+ \frac1{75} \cos 10 y \right)\right| \\
        &\le \frac{\pi ^2}{75} + \sqrt 2 \pi \left\| \varrho + \frac1{75} \cos 10 y \right\| _{L ^2} \le \frac{\pi ^2}{75} + \cO (\e).
    \end{align*}
    Hence, combined with $|B| \le \e$, we get
    \begin{align*}
        |\cS _A| \le \frac2{75} \e + \cO (\e ^2) \le \e.
    \end{align*}
    These bounds are uniform in $\varrho\in\cB_\varrho$: the ball is bounded in $H^2$, and $P$ is polynomial in $B$ and $\varrho$. Next, we focus on the map $\cS _B$. By orthogonality, 
    \begin{align*}
        & \left|\left( e _1 e _2, \varrho \right)\right|, \left|\left( e _2 \cos 10 y, \varrho \right)\right| \le 2 \sqrt 2 \pi \e, \\
        & (P, e _2) = \left(2 \left(e _2 + \varrho \right) ^2, e _2 \right) _{L ^2} + \cO (\e) = \left(2 \left(e _2 - \frac1{75} \cos 10 y \right) ^2, e _2 \right) _{L ^2} + \cO (\e) = \cO (\e).
    \end{align*}
    Therefore, the denominator of \eqref{eqn: B'} is $3 \pm \cO (\e)$, whereas the numerator is bounded by $\e + \frac{4 \sqrt 2}{\pi} \e + \cO (\e ^2)$. Therefore, we have the bound $\left|\cS _B\right| \le \frac{\e + \frac{4 \sqrt 2}\pi \e + C \e ^2}{3 - C \e} < \e$. Clearly, the Lipschitz constant is bounded by $\frac13 + C\e < 1$, so $\cS$ is indeed a contraction.
\end{proof}

\begin{proof}[Proof of \cref{lemma:construction 2}]
    We now study the Lipschitz property of $\varrho \mapsto (A _\varrho, B _\varrho)$.
    With $\e,e_1,e_2$ fixed, write $P(\varrho,B):=P(\e,e_1,e_2,B,\varrho)$. Given $\varrho _1, \varrho _2 \in \cB _\varrho$, the $H^2$ product estimate gives
    \begin{align*}
    \|P(\varrho_1,B _{\varrho _1})-P(\varrho_2, B _{\varrho _2})\|_{L^2}
    \le C\bigl(\|\varrho_1-\varrho_2\|_{\cY}+|B _{\varrho _1}-B _{\varrho _2}|\bigr)
    \end{align*}
    uniformly on the prescribed ball. The equation \eqref{eqn: A'} yields 
    \begin{align*}
        A _{\varrho _1} - A _{\varrho _2} &= -6 B _{\varrho _1} \e \cdot \frac1{2 \pi ^2} ( e _1 \cos 10 y, \varrho _1 - \varrho _2 ) 
        -6\e\cdot\frac1{2\pi^2}(B _{\varrho _1}-B _{\varrho _2})(e _1 \cos 10 y,\varrho_2) \\
        & \qquad - 4 \e \cdot \frac1{2 \pi ^2} ( e _1 ^2, \varrho _1 - \varrho _2 )
        - \e ^2 \cdot \frac1{2 \pi ^2} ( P (\varrho _1, B _{\varrho _1}) - P (\varrho _2, B _{\varrho _2}), e _1 ),\\
       \leadsto \quad |A _{\varrho _1} - A _{\varrho _2}|   &\le C \e \| \varrho _1 - \varrho _2 \| _{\cY} + C \e |B _{\varrho _1} - B _{\varrho _2}| .
    \end{align*}
    As for \eqref{eqn: B'}, the denominator is $\frac6{\sqrt 2 \pi}$-Lipschitz in $\varrho$, and for the numerator, we have 
    \begin{align*}
        & \left| A _{\varrho _2} - A _{\varrho _1} + 4 \cdot \frac1{2 \pi ^2} (e _1 e _2, \varrho _2 - \varrho _1 ) + \e \cdot \frac1{2 \pi ^2} ( P (\varrho _2, B _{\varrho _2}) - P (\varrho _1, B _{\varrho _1}), e _2 ) \right| \\
        & \qquad \le |A _{\varrho _2} - A _{\varrho _1}| + \frac4{\sqrt 2 \pi} \| \varrho _2 - \varrho _1 \| _{L ^2} + C \e \| \varrho _2 - \varrho _1 \| _{\cY} + C \e |B _{\varrho _2} - B _{\varrho _1}|.
    \end{align*}
    Therefore 
    \begin{align*}
        |B _{\varrho _1} - B _{\varrho _2}| &\le \frac1{3 + \cO (\e)} \left(
            |A _{\varrho _2} - A _{\varrho _1}| + \frac4{\sqrt 2 \pi} \| \varrho _2 - \varrho _1 \| _{L ^2} + C \e \| \varrho _2 - \varrho _1 \| _{\cY} + C \e |B _{\varrho _2} - B _{\varrho _1}|
        \right) \\
        & \qquad + \frac1{(3 + \cO (\e)) ^2} \cdot \frac6{\sqrt 2 \pi} \| \varrho _1 - \varrho _2 \| _{L ^2}.
    \end{align*}
    Combined, we conclude after absorbing the terms proportional to $\e(|A _{\varrho _1}-A _{\varrho _2}|+|B _{\varrho _1}-B _{\varrho _2}|)$ for sufficiently small $\e$,
    \begin{align*}
        |A _{\varrho _1} - A _{\varrho _2}| &\le C \e \| \varrho _1 - \varrho _2 \| _{\cY}, & 
        |B _{\varrho _1} - B _{\varrho _2}| &\le \frac1{\sqrt2 \pi} \left(\frac43 + \frac69 + C \e \right) \| \varrho _1 - \varrho _2 \| _{\cY}.
     \end{align*}
    Note that $\frac1{\sqrt2\pi}\left(\frac43+\frac69\right) = \frac{\sqrt2}\pi < 1$; therefore $\varrho \mapsto (A _\varrho, B _\varrho)$ is $1$-Lipschitz.
\end{proof}

Based on the proof of \cref{lem:fixed-point,lemma:construction 1,lemma:construction 2}, we can calculate the $\ye$ and $\Fe$ iteratively.

\begin{remark} 
Assume a power series expansion 
\begin{align*}
    \Ae &= a _0 + \e a _1 + \e ^2 a _2 + \cdots, &
    \Be &= b _0 + \e b _1 + \e ^2 b _2 + \cdots, \\
    \vfe &= \varrho _0 + \e \varrho _1 + \e ^2 \varrho _2 + \cdots, &
    \ye &= \psi _0 + \e \psi _1 + \e ^2 \psi _2 + \cdots,
\end{align*}
with 
\begin{align*}
    \psi _0 &= \cos 5y, & 
    \psi _1 &= \cos (3 x + 4 y) + \varrho _0, & 
    \psi _{i + 1} &= \varrho _i, \qquad i \ge 1.
\end{align*}
The right-hand side of \eqref{eqn: phi} also becomes a series expansion in $\e$, and \eqref{eqn: orthogonality} allows us to solve $a _i$, $b _i$ inductively. We computed the following series expansion using Mathematica\textsuperscript{\upshape\tiny\copyright}:
\begin{align*}
    \Fe (s) &= -25 s + \e T_2 (s) + \e (A _\e s + B _\e T_3 (s)) \\
    &= -25 s + \e (2 s ^2 - 1) + \left(
    \frac{2}{75} \e ^2 
    + \frac{7005344}{71296875} \e ^4
    + \cO (\e ^6)
    \right) s \\
    & \hspace{9.2em} + \left( -\frac{226}{2925} \e ^2 +\frac{731824}{2780578125}\e ^4 
    + \cO (\e ^6)
    \right) (4 s ^3 - 3 s), \\ 
    \ye (x, y)
    & = 
    \cos 5 y + \e \left(
        \cos (3 x + 4 y)
        -\frac1{75} \cos 10 y 
    \right) \\
    & {} \hphantom{{}= \cos 5 y{}} + \e^2 \left(
    \frac2{15} \cos (3 x - y)
    - \frac2{65} \cos (3 x + 9 y)
    + \frac{38}{73125}  \cos 15 y
    \right) + \cO (\e ^3).
\end{align*}
In particular, we have $a _{2 i} = b _{2 i} = 0$ for any $i \in \N$ due to parity. \cref{fig:symmetry} uses this series expansion up to $\e ^7$. One can follow \cite{coti2023stationary} to prove that $\Ae$ and $\Be$ are analytic functions of $\e$. But the exact power series in terms of $\e$, or the analyticity, is not needed for our proof. 
\end{remark}

\begin{remark}
    It is possible that $\ye$ is not mean-zero. Nevertheless, we can shift the stream function by a constant without changing the velocity or the vorticity. Indeed, $\ome = \Delta \ye \in \cX$.
\end{remark}

Based on the explicit form $\Delta \ye = \Fe (\ye)$, and $\|\ye-\psi_0\|_{\cY} \le C\e$, by iterations on elliptic estimates, we have the uniform estimate on the Sobolev space.

\begin{proof}[Proof of \cref{pro:Euler_Steady}]
    Lemma \ref{lem:fixed-point} provides the construction of the steady state $\ye$ and function $\Fe$, up to $\cO (\e ^2)$ error. For every $k \ge 2$, we have
    \begin{align*}
        \| \Delta \ye - \Delta \psi _0 \| _{H ^k} & \le \| \Fe (\ye) - F _0 (\ye) \| _{H ^k} + \| F _0 (\ye) - F _0 (\psi _0) \| _{H ^k} \\
        & \le C \e \left(1 + \| \ye \| _{H ^k} ^3 \right) + 25 \| \ye - \psi _0 \| _{H ^k}.
    \end{align*}
    Here we used that $\Fe$ is a cubic polynomial which is $\e$ away from $F _0$, and $F _0 (s) = -25 s$. Elliptic regularity allows us to show the $H ^{k + 2}$ bound inductively, starting with $\| \ye - \psi _0 \| _{H ^2} \le C \e$ given by Lemma \ref{lem:fixed-point}. Similarly, the $\cO (\e ^2)$ error is bounded by $C \e ^2$ in every $H ^k$ norm. Analyticity follows the semilinear elliptic theory since $\Fe$ is a cubic polynomial with constant coefficients. A quantitative description can be found in \cite{coti2023stationary}*{lemma 2.11}.
\end{proof}

\subsection{The Riesz projection}
\label{sec:riesz}

In this section we show item 2.~in \cref{pro:Euler_spectral}. We begin with the following lemma, which gives an expression for $\cL \Delta \omE$.

\begin{lemma}
    \label{lem:L-Delta-omE} 
    Let $\uE$ be a stationary Euler solution, and $\cL$ be the associated linearize operator \eqref{eq:w_cpt}. Suppose its vorticity $\omE$ and stream function $\yE$ are related by $\omE = F (\yE)$, where $F$ is smooth and has nonvanishing derivative over the range of $\yE$. Then 
    \begin{align*}
        \cL \Delta \omE = -F'' (\yE) \uE \cdot \nabla |\nabla \yE| ^2 = -F'' (\yE) \uE \cdot \nabla |\uE| ^2.
    \end{align*}
\end{lemma}

\begin{proof}
    Recall $\omE = F (\yE) = \Delta \yE$, so $\nabla \omE = F' (\yE) \nabla \yE$, and 
    \begin{align*}
        \Delta \omE = \div (\nabla \omE) = F' (\yE) \Delta \yE + F'' (\yE) |\nabla \yE| ^2 = F' (\yE) \omE + F'' (\yE) |\nabla \yE| ^2.
    \end{align*}
    Therefore, 
    \begin{align}
        \label{eqn: HDomE}
        \cH \Delta \omE = -\frac1{F' (\yE)} \Delta \omE + \omE = -\frac{F'' (\yE)}{F' (\yE)} |\nabla \yE| ^2,
    \end{align}
    and 
    \begin{align*}
        \cL \Delta \omE = \cJ \left( 
            -\frac{F'' (\yE)}{F' (\yE)} |\nabla \yE| ^2
        \right) = -F'' (\yE) \uE \cdot \nabla |\nabla \yE| ^2 = -F'' (\yE) \uE \cdot \nabla |\uE| ^2.
    \end{align*}
    This completes the proof of the desired identity.
\end{proof}

\begin{remark}
    \label{rem:eigenfunction-bad}
    For shear flow or cellular flow,  let $\phi ^*$ be an eigenfunction of $\cL ^*$ with nonzero eigenvalue $\lambda $, we have $( \Delta \omE, \phi ^* ) = 0$, because
    \begin{align*}
        ( \Delta \omE, \phi ^* ) = \frac1\lambda ( \Delta \omE, \cL ^* \phi ^* ) = \frac1\lambda \left( 
            \cL \Delta \omE, \phi ^*
        \right). 
    \end{align*}
    For pressure-less flows such as a shear flow $u _0 = U (y) e _1$ or a radial flow $u _0 = U (r) e _\phi$, $\cL \Delta \omE$ is zero because $|u _0| = $ constant along streamlines. For the cellular flow, $F _0$ is a linear function, so $F _0 '' = 0$, and $\cL \Delta \omE$ vanishes as well.
\end{remark}

We would like to perturb a shear flow $\omega _0 = -U' (y)$ to something else $\ome$ such that $\cL _\e \Delta \ome$ and the inner product $( \Delta \ome, \fe ^* )$ are nonzero. Unfortunately, for a Kolmogorov flow, it is both pressure-less and with linear $F$, so if we perturb it with an $\e$ perturbation, $\cL _\e \Delta \ome$ and the inner product will still vanish at $\cO (\e)$ order. We could nevertheless hope the inner product stays nontrivial at $\cO (\e ^2)$ order. It is also important that the perturbation $\ome$ cannot stay in the kernel of $\cL _0 = -U \partial _x (\I + 25 \Delta ^{-1})$, since otherwise either $F$ remains an affine function and $F '' = 0$, or $\ue$ is another shear flow \footnote{
If $\ome \in \ker \cL _0$ then $\partial _x (\Delta + 25) \ye = 0$, so $(\Delta + 25) \ye = f (y)$ is a function of $y$, and $\ye = \tilde \psi _\e + g (y)$ for some $\tilde \psi _\e \in \ker (\Delta + 25)$ and $g (y) = (\Delta + 25) ^{-1} f (y)$ is another function of $y$. The semilinear equation it satisfies is $F (\ye) = \Delta \ye = -25 \tilde \psi _\e + g '' (y)$. If $\tilde \psi _\e$ is nontrivial then taking partial derivative in $x$ shows $F ' (\ye) = -25$, so $F'' (\ye) = 0$. If $\tilde \psi _\e = 0$ then $\ye = g (y)$ is a shear flow.
}, and both cases annihilate $\cLe \Delta \ome$. Therefore, we need a perturbation in the style of \cite{coti2023stationary}. Note that the examples coming from a direct application of the method in \cite{coti2023stationary} would not produce the example we want.

\begin{remark}
\label{rem:rescale-bad}
In \cite{coti2023stationary}, the authors study the non-shear steady state near the Kolmogorov flow $(\sin{y},0) ^\top$. The stream function $\tilde \psi _\e$ of the associated steady state is even-even symmetric and solves a semilinear elliptic equation $\Delta \tilde \psi _\e = -\tilde \psi _\e + \e (\Ae \tilde \psi _\e + \Be \tilde \psi _\e ^3 + \frac15 \tilde \psi _\e ^5)$. However, the Kolmogorov flow has no unstable modes. The natural rescaling $\ye (x, y) := \tilde \psi _\e (5x, 5y)$ then gives rise to unstable steady states near our sinusoidal flow $(5\sin{5y}, 0) ^\top$. Unfortunately, the forcing associated with such a steady state has trivial projection to the most unstable mode. Below, we show $(\Delta \ome, \cHe \fe ^-) = 0$ by proving $\cHe \Delta \ome$ and a most stable mode $\fe ^-$ have disjoint frequency support. 

Since $\ye$ is rescaled from $\tilde \psi _\e$, we know $\ye$ is even-even symmetric and $\frac{2\pi}5$-periodic in both $x$ and $y$ directions, and this symmetry is inherited by $\cHe \Delta \ome$ due to \eqref{eqn: HDomE}. Therefore its frequency is supported in $\Z(5,0) + \Z(0,5) = (5 \Z) ^2$:
\begin{align*}
    \cF _i:= \left\{ 
        (m, n) \in \Z ^2 \mid \int _{\T ^2} \cHe \Delta \ome \cos{mx} \cos{ny} \d x \d y \ne 0 
    \right\} \subset \Z(5,0) + \Z(0,5).
\end{align*}
Moreover, \cite{coti2023stationary} implies $\tilde \psi _\e$ has a power series expansion $\tilde \psi _\e = \sum _{i = 0} ^\infty \e ^i \tilde \psi _i$, so $\psi _i (x, y) = \tilde \psi _i (5 x, 5 y)$, hence each $\psi _i$ is even-even symmetric and its wavenumbers are also multiple of 5. 

Next, we show the Fourier support of $\fe ^-$ is disjoint from $5 \Z \times 5 \Z$. Similar to $\ye$, the linearized Euler operator also has a series expansion $\cLe = \sum _{i = 0} ^\infty \e ^i \cL _i$,
and each $\cL _i$ is $\frac{2 \pi}5$-periodic in both direction.
Moreover, all the most stable eigenvalues and their associated stable modes of $\cLe$ can be written into $\e$ expansion. Let $\lae := \sum _{i = 0} ^\infty \e ^i \lambda _i$ be a most stable eigenvalue and $\fe ^- := \sum _{k = 0} ^\infty \e^ k \phi _i ^-$ be a most stable mode, then $ \cLe \fe ^- = \lae \fe ^-$ gives a system of inductive equations on $\lambda _i$ and $\phi _i$: $\sum _{i + j = k} (\cL _i - \lambda _i) \phi _j ^- = 0$. For example, at $k = 0$, we have $(\cL _0 - \lambda _0) \phi _0 ^- = 0$, and in particular $\phi _0 ^-$ is the most stable mode for the sinusoidal flow $(5 \sin{5y}, 0) ^\top$. Now, decompose $\phi _i ^-$ by parity:
\begin{equation}
    \phi ^- _i := \phi _i ^{oo} + \phi _i ^{oe} + \phi _i ^{eo} + \phi _i ^{ee},
\end{equation}
where $o\_$ ($e\_$) means odd (even) in $x$ and  $\_o$ ($\_e$) means odd (even) in $y$. By \cites{friedlander1998instability,MeshalkinSinai1961}, we know the frequency support of $\phi _0 ^{ee}$ lies in $(\pm1, 0) + \Z(5,0) + \Z(0,5)$. By the inductive equation, we can show this is also true for general $\phi _i ^{ee}$, $i \ge 0$. Therefore, the Fourier support of $\fe ^{ee}$ is supported in $(\pm1, 0) + \Z(5,0) + \Z(0,5)$.

By the above calculation, we have $(\Delta \ome, \cHe \fe ^-) = (\cHe \Delta \ome, \fe ^-) = (\cHe \Delta \ome, \fe ^{ee}) = 0$ by orthogonality, and thus by \cref{lem:riesz} the static forcing has trivial projection to the most unstable modes.
\end{remark}

Directly rescaling the solution from \cite{coti2023stationary} fails. However, if we instead restrict ourselves to the symmetry class $\cX$ and create a new steady state by \cref{pro:Euler_Steady}, then for $0 < \e \ll 1$, the projection forcing term $\P _{\lambda _\e} (\Delta \ome)$ is nonzero.

\begin{proof}[Proof of Proposition \ref{pro:Euler_spectral}]

Item~1.~and item~3.~are proven in \cref{sec:spec-euler} following standard spectral analysis.  
It remains to show that 
\begin{equation}
    \lim _{\e \to 0} \frac1{\e ^2} \P _{\lambda _\e} (\Delta \ome) = \lim _{\e \to 0} \frac1{\e ^2} \frac{( \Delta \ome, \fe ^* )}{( \fe, \fe ^* )} \fe \neq 0.
\end{equation}
By \cref{rem:eigenfunction-bad}, this is equal to 
\begin{equation}
    \lim _{\e \to 0} \frac1{\e ^2} \P _{\lambda _\e} (\Delta \ome) = \lim _{\e \to 0} \frac1{\e ^2} \frac1{\lambda _\e} \frac{( \cLe \Delta \ome, \fe ^* )}{( \fe, \fe ^* )} \fe = \left ( \lim _{\e \to 0} \frac1{\e ^2} \cLe \Delta \ome, \phi _0 ^* \right) \cdot \frac{\phi _0}{\lambda _0 ( \phi _0, \phi _0 ^* )}.
\end{equation}
We have used \cref{sec:spectral_Euler} to justify the convergence of eigenvalues and eigenfunctions in $L ^2$. Now using \cref{lem:L-Delta-omE} we compute 
\begin{align*}
    G := \lim _{\e \to 0} \frac1{\e ^2} \cLe \Delta \ome = -\lim _{\e \to 0} \frac1{\e ^2} F'' (\ye) \ue \cdot \nabla |\ue| ^2 = \lim _{\e \to 0} \frac{-F'' (\ye)}\e \frac{\ue \cdot \nabla |\ue| ^2}\e.
\end{align*}
With 
\begin{align*}
    \Fe &= F _0 + \e F _1 + \cO (\e ^2), &
    \ue &= u _0 + \e u _1 + \cO (\e ^2),
\end{align*}
we have $\lim_{\e \rightarrow 0} \dfrac{\Fe''(\ye)}{\e}=\lim_{\e\rightarrow 0}\dfrac{\Fe''(\ye)-F_0''(\ye)}{\e}=F_{1}''(\psi_0)$, and 
\begin{equation}
    \lim_{\e\rightarrow 0} \frac{\ue\cdot\ \nabla|\ue|^2}{\e}=\lim_{\e\rightarrow 0}\frac{\ue\cdot \nabla|\ue|^2-u_{0}\cdot \nabla|u_{0}|^2}{\e}=2 (u_0\cdot \nabla u_0\cdot u_1+u_0\cdot \nabla u_1\cdot u_0+u_1\cdot \nabla u_0 \cdot u_0).
\end{equation}
Therefore,
\begin{align}
    \label{eqn: non-vanishing}
    G = \lim _{\e \to 0} \frac1{\e ^2} \cLe \Delta \ome =  
    -2 F _1'' (\psi _0) \left(
        u _0 \cdot \nabla u _0 \cdot u _1 + u _0 \cdot \nabla u _1 \cdot u _0 + u _1 \cdot \nabla u _0 \cdot u _0
    \right),
\end{align}
where the limit holds in the sense of uniform convergence.
For a shear flow $u _0 = U (y) e _1$, so
\begin{align*}
    G = -2 F _1'' (\psi _0) \left[
        U ^2 \partial _x (u _1 \cdot e _1) + (u _1 \cdot e _2) U U'
    \right].
\end{align*}
Recall $u _1 = (-\partial _y \psi _1, \partial _x \psi _1)$, we can write it as 
\begin{align*}
    G = - 2 F _1'' (\psi _0) \left[
        -U ^2 \partial _x \partial _y \psi _1 + \partial _x \psi _1 U U'
    \right] = 2 F _1'' (\psi _0) U \partial _x \left[
        U \partial _y \psi _1 - \psi _1 U'
    \right].
\end{align*}
Let us now take inner product with $\phi _0 ^*$ and integrate by part:
\begin{align*}
    \left( G, \phi _0 ^* \right) &= \int 2 F _1'' (\psi _0) U \partial _x \left[
        U \partial _y \psi _1 - \psi _1 U'
    \right] \phi _0 ^* \d x \d y \\
    & = - \int 2 F _1'' (\psi _0) \left[
        U \partial _y \psi _1 - \psi _1 U'
    \right] U \partial _x \phi _0 ^* \d x \d y.
\end{align*}
Since $U \partial _x \phi _0 ^* = \frac1{F' _0 (\psi _0)} \cJ _0 \phi _0 ^* = -\frac1{F' _0 (\psi _0)} \lambda _0 \phi _0 ^-$, we conclude 
\begin{align*}
    \left( G, \phi _0 ^* \right) = 2 \lambda _0 \int \frac{F _1'' (\psi _0)}{F _0 ' (\psi _0)} \left[
        U \partial _y \psi _1 - \psi _1 \partial _y U
    \right] \cdot \phi _0 ^- \d x \d y.
\end{align*}
In our construction, $F _0' = -25$, $F _1'' = 4$ and $\psi _1 = \cos (3x + 4y) - \frac1{75} \cos 10 y$.
Note that $\phi _0 ^- = -\frac1{\lambda _0} \cJ _0 \phi _0 ^* = \partial _x (\frac{25U}{\lambda_0}\phi_0^{*} )$ has no zero mode in $x$, so 
\begin{align*}
    \left( G, \phi _0 ^* \right) = -\frac{8 \lambda _0}{25} \int \left[
        U \partial _y \cos (3x + 4y) -  \cos (3x + 4y) \partial _y U
    \right] \cdot \phi _0 ^- d x d y.
\end{align*}
Now we plugin $U = 5 \sin 5 y$ and we get 
\begin{align*}
    U \partial _y \cos (3x + 4y) - \cos (3x + 4y) \partial _y U &= -20 \sin 5 y \sin (3 x + 4 y) - 25 \cos 5 y \cos (3 x + 4 y) \\
    &= -\frac52 \cos (3 x + 9 y) - \frac{45}2 \cos (3x - y).
\end{align*}
From our computation, $\phi _0 ^-$ contains $\cos (3x - y)$ Fourier mode, but does not have $\cos (3x + 9y)$ Fourier mode, so 
\begin{align*}
    \left( G, \phi _0 ^* \right) = -\frac{8 \lambda _0}{25} \cdot \left(-\frac{45}2\right) ( \cos (3x - y), \phi _0 ^- ) = \frac{36 \lambda _0}5 ( \cos (3x - y), \phi _0 ^- ).
\end{align*}
thus 
\begin{align*}
    \lim _{\e \to 0} \frac1{\e ^2} \P _{\lambda _\e} (\Delta \ome) = \frac{36 ( \cos (3x - y), \phi _0 ^- )}{5 ( \phi _0, \cH \phi _0 ^- )} \phi _0.
\end{align*}
Therefore, we have 
\begin{align*}
    \P _{\lambda _\e} (\Delta \ome) = \frac{36 ( \cos (3x - y), \phi _0 ^- )}{5 ( \phi _0, \cH_0 \phi _0 ^- )} \e ^2 \phi _0 + o (\e  ^2) \ne 0.
\end{align*}
The coefficient of $\e ^2$ is non-zero because of \eqref{Kol_equality}. This completes the proof of the proposition.
\end{proof}

\section{Bounds on the Grenier iteration}\label{sec:Grenier}

We now verify the inductive Sobolev bound used in the proof of the main theorem.

\begin{proof}[Proof of Proposition \ref{pro:reg}]
We prove them inductively, by solving
\begin{align*}
    \omega _k (t) = \int _0 ^t \expo{(t - s) \cL} \Delta \omega _{k - 1} - \sum _{\substack{i, j \ge 1 \\i + j = k}} \expo{(t - s) \cL} (u _i \cdot \nabla \omega _j) \d s.
\end{align*}
Note that $\| \expo{t \cL} \| _{\cX ^k \to \cX ^k} \le C _k \expo{\lambda t}$ by Item \ref{item:Euler_spectral_3} of Proposition \ref{pro:Euler_spectral}.

At $k = 1$, we have
\begin{align*}
    \| \omega _1 (t) \| _{H ^{2 m}} \le \int _0 ^t \expo{\lambda (t - s)} \| \Delta \omE \| _{H ^{2 m}} \d s \le C _1 \| \omE \| _{H ^{2 m + 2}} (\expo{\lambda t} - 1) \le C _1 \expo{\lambda t}.
\end{align*}
For $k \ge 2$, assume the bound holds for every $j < k$, then 
\begin{align*}
    \| \omega _k (t) \| _{H ^{2 (m + 1 - k)}} &\le \int _0 ^t \expo{\lambda (t - s)} \left( \| \Delta \omega _{k - 1} \| _{H ^{2 (m + 1 - k)}} + \sum _{i + j = k} \| u _i \cdot \nabla \omega _j \| _{H ^{2 (m + 1 - k)}} \right) \d s \\
    & \le C _k \int _0 ^t \expo{\lambda (t - s)} \left( \| \omega _{k - 1} \| _{H ^{2 (m + 2 - k)}} + \sum _{i + j = k} \| u _i \| _{H ^{2 (m + 1 - k)}} \| \nabla \omega _j \| _{H ^{2 (m + 1 - k)}} \right) \d s \\
    & \le C _k \int _0 ^t \expo{\lambda (t - s)} \left( \| \omega _{k - 1} \| _{H ^{2 (m + 2 - k)}} + \sum _{i + j = k} \| \omega _i \| _{H ^{2 (m + 1 - i)}} \| \omega _j \| _{H ^{2 (m + 1 - j)}} \right) \d s.
\end{align*}
Here we used $H ^{2 (m + 1 - k)} \hookrightarrow H ^4 \hookrightarrow L ^\infty$, and $i, j \le k - 1$. By inductive assumptions, we have
\begin{align*}
    \| \omega _k (t) \| _{H ^{2 (m + 1 - k)}} 
    & \le C _k \int _0 ^t \expo{\lambda (t - s)} \left( 
        \expo{(k - 1) \lambda s} 
        + \expo{k \lambda s}
    \right) \d s \\
    & \le 2 C _k \expo{\lambda t} \int _0 ^t \expo{(k - 1) \lambda s}  \d s \\
    & = \frac{2 C _k}{(k - 1) \lambda} \expo{\lambda t} (\expo{(k - 1) \lambda t} - 1) \\ 
    & \le C _k \expo{k \lambda t}.
\end{align*}
This proves the first inequality.

Now testing \eqref{eqn:remainder} by $\tilde{\omega}$. There exists a positive constant $\tilde C$ such that we have 
\begin{align*}
    \frac{\d}{\d t} \|\tilde \omega\| _{L ^2} &\le 2 \sum _{k = 0} ^{m - 1} \nu ^k \| \nabla \omega _k \| _{L^{\infty}} \|\tilde \omega\| _{L ^2} + 2 \nu ^m \| \Delta \omega _{m - 1} \| _{L ^2} + C \sum _{k = m} ^{2 m} \nu ^k \expo{k \lambda t} \\
    &\le 2 \sum _{k = 0} ^{m - 1} \nu ^k \| \nabla \omega _k \| _{L^{\infty}} \|\tilde \omega\| _{L ^2} + \tilde C \sum _{k = m - 1} ^{2 m} \nu ^k \expo{k \lambda t}.
\end{align*}
For $t \le \frac1\lambda \log \left(\frac\delta\nu\right)$, $t$ and $\nu$ satisfies $\nu \expo{\lambda t} \le \delta$. During this time span, we have 
\begin{align*}
    \| \nabla \omega _k \| _{L ^\infty} \le \| \omega _k \| _{H ^3} \le C _k \expo{k \lambda t},
\end{align*}
thus 
\begin{align*}
    2 \sum _{k = 0} ^{m - 1} \nu ^k \| \nabla \omega _k \| _{L^{\infty}} \le 2 \| \nabla \omE \| _{L ^\infty} + \sum _{k = 1} ^{m - 1} C _k \delta ^k \le 4 \| \nabla \omE \| _{L ^\infty}.
\end{align*}
In the last step, we chose $\delta \ll 1$ small so that $\sum _{k = 1} ^{m - 1} C _k \delta ^k \le 2 \| \nabla \omE \| _{L ^\infty}$. By our selection of $m$ in \eqref{choice_m}, we have that for sufficiently small $\e > 0$, it holds that
\begin{align*}
    (m - 2) \lambda _0 \ge 1000 = 8 \| \nabla \omega _0 \|_{L ^{\infty}} \quad \leadsto \quad (m - 2) \lambda _\e \ge 4 \| \nabla \ome \|_{L ^{\infty}} \text{ as } \lambda _\e \to \lambda _0, \omega _\e \to \omega _0 \text{ in } H ^3.
\end{align*}
The energy inequality is now simplified to
\begin{align*}
    \frac{\d}{\d t} \|\tilde \omega\| _{L ^2} \le (m - 2) \lambda \| \tilde \omega \| _{L ^2} +  \tilde C \sum _{k = m - 1} ^{2 m} \nu ^k \expo{k \lambda t}.
\end{align*}
Solving this differential inequality gives
\begin{equation}
    \| \tilde{\omega} (t) \| _{L ^2} \le \tilde C \sum _{k = m - 1} ^{2 m} \frac{\nu ^k}{(k - m + 2) \lambda} \left( 
        \expo{k \lambda t} - \expo{(m - 2) \lambda t}
    \right) \le \tilde C \sum _{k = m - 1} ^{2 m} \nu ^k \expo{k \lambda t}.
\end{equation}
This, together with the $H ^{2 (m + 1 - k)}$ bound of $\omega _k$, completes the proof of the proposition.
\end{proof}

\begin{remark}
    This is a time span during which the energy has NOT been noticeably dissipated. Indeed, by the maximum principle $\| \omnu (t,\cdot )\| _{L ^{\infty}} \leq \| \omnu (0,\cdot )\| _{L ^{\infty}}= \cO (1)$ during this time span, so the total energy dissipation is 
    \begin{align*}
        \nu \int _0 ^t \int |\nabla \unu| ^2 \d x \d s = \nu \int _0 ^t \int |\omnu| ^2 \d x \d s = \cO (\nu \log (1/\nu)).
    \end{align*}
    Therefore $\| \unu \| _{L ^2}$ is still comparable with $\| \uE \| _{L ^2}$ with $o (1)$ energy dissipation.
\end{remark}

\section*{Acknowledgment}
S.~H.~is supported by NSF DMS-2406293.
Y.~H.~acknowledges partial funding from EPSRC Horizon Europe Guarantee EP/X020886/1.
Y.~H.~and J.~Y.~would like to acknowledge the support of the Institut Henri Poincar\'e (UAR 839 CNRS-Sorbonne Universit\'e), and LabEx CARMIN (ANR-10-LABX-59-01). 

\section*{Disclosure of AI usage}
We used ChatGPT 5.5 to understand the theoretical background of the classical spectral theory behind Proposition \ref{pro:Euler_spectral}. We also applied ChatGPT 6 to proofread the finished draft. 
\appendix
\crefalias{section}{appendix}
\crefalias{subsection}{appendix}

\section{Lower bound on the separation time}
\label{sec:lower-bound}

In this section, we estimate the speed of separation by providing the lower bound on the separation time for stationary Euler solution. It is based on a direct energy estimate.

\begin{proposition}
    Let $\uE \in C ^1$ be a non-constant stationary Euler solution. Define the average speed $U$ and maximal stretching rate $S$ by
    \begin{align*}
        U ^2 &= \frac1{|\T ^2|} \int _{\T ^2} |\uE| ^2 \d x \d y, & 
        S &= \max _{(x, y) \in \T ^2} \|D \uE (x, y)\|.
    \end{align*}
    Here $\| \cdot \|$ represents the matrix norm, and $D \uE = \frac12(\nabla \uE + (\nabla \uE) ^\top)$ is the symmetric part of velocity gradient. Then 
    \begin{align*}
        T _\nu ^\delta (\uE) \ge \frac1{2 S} \log \left(\frac {U ^2}{S} \cdot \frac{\delta ^2}\nu \right).
    \end{align*}
\end{proposition}

\begin{proof}
Let us denote the difference
\begin{align*}
    v &= \unu - \uE, &
    q &= \pnu - \pE, &
    w &= \omnu - \omE.
\end{align*}
Subtracting \eqref{eq:euler} from \eqref{eq:ns}, one obtains
\begin{align*}
    \partial_t v + \unu \cdot \nabla v + v \cdot \nabla \uE + \nabla q = \nu \Delta \unu,
    \qquad
    \nabla \cdot v = 0.
\end{align*}
Now we multiply by $v$ and derive the energy estimate. For the transport and pressure term, since $\uE$ and $v$ are divergence free and $\nabla q$ is curl free, we have
\begin{align*}
    \int_{\T ^2} \unu \cdot \nabla v \cdot v \d x \d y=0, \qquad
    \int_{\T ^2} \nabla q \cdot v \d x \d y=0.
\end{align*}
For the stretching term, we have 
\begin{align*}
    \int_{\T ^2} v \cdot \nabla \uE \cdot v \d x \d y = \int_{\T ^2} (v\otimes v) : \nabla \uE \d x \d y = \int_{\T ^2} (v\otimes v) : D \uE \d x \d y.
\end{align*}
Here recall $D u ^E$ is the symmetric part of $\nabla u ^E$. For the dissipation term, $\Delta \unu = \Delta \nabla ^\perp \psi ^\nu = \nabla ^\perp \omnu$, where $\psi ^\nu$ is the stream function. After applying integration by parts, we have
\begin{align*}
    \int_{\T ^2} v \cdot \Delta \unu \d x \d y &= -\int_{\T ^2} w \omnu \d x \d y = -\int _{\T ^2} w ^2 \d x \d y - \int _{\T ^2} w \omE \d x \d y \le \frac14 \int _{\T ^2} |\omE| ^2 \d x \d y.
\end{align*}
In summary, we have the energy estimate
\begin{align*}
    \frac{\d}{\d t}\int_{\T ^2} \frac{|v|^2}2 \d x \d y + \int_{\T ^2} (v\otimes v) : D \uE \d x \d y \le \frac\nu4\int_{\T ^2} |\omE|^2\d x \d y.
\end{align*}
By our definition of the stretching rate $S$, we have 
\begin{align*}
    \frac{\d}{\d t}\int_{\T ^2}\frac{|v|^2}2\d x \d y\le\int_{\T ^2}S |v| ^2\d x \d y+\frac\nu4\int_{\T ^2} |\omE|^2\d x \d y.
\end{align*}
Indeed, $\int_{\T ^2} |\omE|^2\d x \d y$ is constant in time for 2D Euler, even for nonstationary solutions. By the equal partition of enstrophy between the symmetric and skew-symmetric parts, we can bound the enstrophy by
\begin{align*}
    \| \omE \| _{L ^2 (\T ^2)} ^2 = 2 \int _{\T ^2} D \uE : D \uE \d x \d y \le 4 S ^2 |\T ^2| .
\end{align*}
Therefore, by Gr\"onwall inequality, we conclude
\begin{align*}
    \|v(t)\|_{L^2(\T)}^2\le\expo{2 S t}\left(\|v(0)\|_{L^2(\T)}^2+\nu S |\T ^2| \right) = \nu S |\T ^2| \expo{2 S t}.
\end{align*}
If at $t = T _\nu ^\delta$ we achieve $\| v (t) \| _{L ^2} ^2 = \delta ^2 \| \uE \| _{L ^2} ^2 = \delta^2 U ^2 |\T ^2|$, then 
\begin{align*}
    T _\nu ^\delta \ge \frac1{2 S} \log \left(\frac {\delta ^2 U ^2}{S \nu}\right).
\end{align*}
Thus, we have completed the proof of this proposition.
\end{proof}

\begin{remark}
    In the case of nonstationary flow, the velocity gradient does not need to stay constant, and we should replace $\expo{2 S t}$ in the Gr\"onwall inequality by $\expo{2 S (t)}$ with the accumulative stretching:
    \begin{align*}
        S (t) = \int _0 ^t \| D \uE (s) \| _{L ^\infty (\T)} \d s.
    \end{align*}
    Since velocity gradient grows at most exponentially (see \cite{LeiShi2016}*{eq. (2.1)}), via similar argument we conclude 
    \begin{align*}
        \|v(t)\|_{L^2(\T)}^2 \le \nu \cdot C \expo{C \expo {C t}}.
    \end{align*}
    As a consequence, we have the following bound as $\nu \to 0$:
    \begin{align*}
        T _\nu ^\delta \ge \frac1C \log \left(\frac1C \log \left(\frac{\delta ^2}{C \nu}\right) \right).
    \end{align*}
    The same bound was established in \cite{Chemin1996} within Yudovich class. Whether this lower bound can be achieved is left for future investigation. Indeed, to achieve this rate, one would need a nonstationary Euler flow with exponentially growing velocity gradient, which is currently unknown.
\end{remark}

\section{Properties of the linearized Euler operator} 
\label{sec:spectral_Euler}  

In this section, we collect the proof of certain lemmas and cite useful results on the spectrum theory, and develop all the necessary spectral properties of the linearized Euler operator $\cL$. We will start by verifying the functional settings of defined linear operators in \cref{sec:functional}, and then  we list property of the linear opearator based on Hamiltonian structure in \cref{sec:Hamiltonian} for general stationary Euler flow $\uE$. For our perturbed velocity field $\ue$, we start with developing the Lyapunov exponent control in \cref{sec:lyapunov}, which is an essential ingredient to study the right-most unstable eigenvalue $\lambda_\e$. Then, we analyze the spectrum, resolvent, and semigroup associated with the transport operator $\cTe$ in \cref{sec:spec-transport} and linearized Euler operator $\cLe$ in \cref{sec:spec-euler}, which would provide the estimates stated in \cref{pro:Euler_spectral}. The ideas follow the results in \cites{ShvydkoyLatushkin2005,ShvydkoyFriedlander2008,lin2013unstable}. For completeness, we give a self-contained proof.

\subsection{Functional settings}
\label{sec:functional}

Let us first verify the properties of linearized Euler operator $\cL$ \eqref{eq:w_cpt} , $\cT$, $\cK$ in the decomposition  \eqref{eq:w_cpt}, and $\cJ$ , $\HX$ in the composition \eqref{eqn:defn-J-H}, in our Hilbert space $\cX$ \eqref{def:X}.

\paragraph{Properties of $\cT$.}

Recall that $\cT = -\uE \cdot \nabla$ given in \eqref{eq:w_cpt}, and $\uE$ is smooth and bounded. The following properties hold for $\cT$.
\begin{itemize}
    \item $\cT$ is densely defined. This is because for any $f \in H ^1 (\T ^2)$ we have $\cT f \in L ^2$, so 
    \begin{align*}
        H ^1 (\T ^2) \cap \cX \subset D (\cT) := \{f \in \cX \mid \cT f \in \cX \text{ in the sense of distributions}\}.
    \end{align*}
    By density of $H ^1$ in $L ^2$, we conclude $D (\cT)$ is also dense in $\cX$.

    \item $\cT$ is skew-adjoint. For any $f, g \in H ^1 (\T ^2)$ it holds 
    \begin{align*}
        (\cT f, g) = \int _{\T ^2} (-\uE \cdot \nabla f) \bar g \d x \d y = \int _{\T ^2} f \cdot (\uE \cdot \nabla \bar g) \d x \d y = -(f, \cT g).
    \end{align*}
    By density, the above identity holds for all $f, g \in D (\cT)$. Thus $\cT ^* = -\cT$.

    \item $\cT$ is a closed operator. Take a sequence $f _k \in D (\cT)$ and assume $f _k \to f \in L ^2$, $\cT f _k \to g \in L ^2$. Then for any test function $h \in C ^\infty (\T ^2)$, we have 
    \begin{align*}
        \int _{\T ^2} \cT f \cdot \bar h \d x \d y = (f, \cT ^* h) = \lim _{k \to \infty} (f _k, \cT ^* h) = \lim _{k \to \infty} (\cT f _k, h) = (g, h) \le \| g \| _{L ^2} \| h \| _{L ^2}.
    \end{align*}
    Hence $\cT f \in L ^2$, $f \in D (\cT)$ and $\cT f = g$.
\end{itemize}

\paragraph{Property of $\cK$.}

Recall that $\cK = -\nabla \omE \cdot \curl ^{-1}$ defined in \eqref{eq:w_cpt}, and $\omE$ is smooth and bounded. The following properties hold for $\cK$.

\begin{itemize}
    \item $\cK$ is compact. This is because 
    \begin{align*}
        \| \cK f \| _{H ^1} \le C \| \nabla \omE \| _{W ^{1, \infty}} \left \| \curl ^{-1} f \right \| _{H ^1} \le C \| \nabla \omE \| _{W ^{1, \infty}} \| f \| _{L ^2}.
    \end{align*}
    Thus $\cK$ is bounded from $L ^2$ to $H ^1$. Because $H ^1$ is compactly embedded in $L ^2$, we have $\cK$ is a compact operator.
\end{itemize}

\paragraph{Property of $\cJ$.} 

Recall that $\cJ = F' (\yE) \uE \cdot \nabla$ given in \eqref{eqn:defn-J-H}, where $\yE$ is smooth and bounded, and $F$ is smooth and bounded with nonvanishing derivative in the range of $\yE$. The following properties hold for $\cJ$.

\begin{itemize}
    \item $\cJ$ is densely defined. This is because for any $f \in H ^1 (\T ^2)$ we have $\cJ f \in L ^2$, so 
    \begin{align*}
        H ^1 (\T ^2) \cap \cX \subset D (\cJ) := \{f \in \cX \mid \cJ f \in \cX \text{ in the sense of distributions}\}.
    \end{align*}
    By density of $H ^1$ in $L ^2$, we conclude $D (\cJ)$ is also dense in $L ^2$.

    \item $\cJ$ is skew-adjoint. Note that 
    \begin{align*}
        \uE \cdot \nabla (F' (\yE)) = F'' (\yE) \uE \cdot \nabla \yE = F'' (\yE) \uE \cdot (-\uE) ^\perp = 0.
    \end{align*}
    Similar to $\cT$, for any $f, g \in H ^1 (\T ^2)$ it holds 
    \begin{align*}
        (\cJ f, g) = \int _{\T ^2} F' (\yE) (\uE \cdot \nabla f) \bar g \d x \d y = -\int _{\T ^2} F' (\yE) f \cdot (\uE \cdot \nabla \bar g) \d x \d y = -(f, \cJ g).
    \end{align*}
    By density, the above identity holds for all $f, g \in D (\cJ)$. Thus $\cJ ^* = -\cJ$.

\end{itemize}

\paragraph{Properties of $\HX$.}

Recall that $\cH = -\frac1{F' (\yE)} + \Delta ^{-1}$ given in \eqref{eqn:defn-J-H}, and $\HX = \P _{\ne 0} \cH$ is its mean-free projection. The following properties hold for $\HX$.

\begin{itemize}
    \item $\HX$ is bounded. Here $F'$ is assumed to be nonvanishing over the range of $\yE$, so $\frac1{F' (\yE)}$ is a bounded continuous function. Moreover, on mean zero subspace of $L ^2 (\T ^2)$, $\Delta ^{-1}$ is bounded. Therefore, $\cH: L ^2 \to L ^2$ is bounded, and thus $\HX: \cX \to \cX$ is bounded. 

    \item $\HX$ is self-adjoint. This is because both $\Delta ^{-1}$ and multiplication by a real smooth function is self-adjoint in $L ^2$, thus $\cH$ are self-adjoint in $L ^2$, and $\HX$ is self-adjoint in $\cX$.
\end{itemize}

\paragraph{Properties of $\cL$.}

Finally, we combine the above conclusions, and show properties of $\cL$.

\begin{itemize}
    \item $\cL$ is densely defined. This is because $\cL = \cT + \cK$ where $\cT$ is densely defined and $\cK$ is bounded. So $D (\cL) = D (\cT)$ is dense in $\cX$.

    \item $\cL$ is a closed operator. Take a sequence $f _k \in D (\cL) = D (\cT)$ and assume $f _k \to f \in L ^2$, $\cL f _k \to g \in L ^2$. Then $\cT f _k = \cL f _k - \cK f _k \to g - \cK f$ in $L ^2$, by the continuity of $\cK$. Since $\cT$ is closed, we have $f \in D (\cT) = D (\cL)$, and $\cT f = g - \cK f \in L ^2$, hence $\cL f = g \in L ^2$ as well. Therefore $\cL$ is also a closed operator.
\end{itemize}

\subsection{Hamiltonian structure}\label{sec:Hamiltonian}

We start with providing a proof for Lemma \ref{lem:riesz}, which gives the spectral projection of a Hamiltonian operator.

\begin{proof}[Proof of Lemma \ref{lem:riesz}]
    In the above, we have verified that $\cL = \cJ \circ \HX$ is a densely defined closed operator, where $\cJ$ is skew adjoint and $\HX$ is bounded and self adjoint. 

    Now suppose $\lambda$ is a nonzero simple eigenvalue of $\cL$. By \cite{kato1995perturbation}*{section III.5}, since $\cL$ is densely defined and closed, the rank of $\P _\lambda$ equals to the algebraic multiplicity of $\lambda$, which is 1. Since $\P _\lambda \phi = \phi$, we know the range of $\P _\lambda$ is spanned by $\phi$ and 
    \begin{align*}
        \P _\lambda f = L (f) \phi, \qquad \forall f \in \cX.
    \end{align*}
    Now we determine the linear functional $L (f)$. By \cite{kato1995perturbation}*{section III.6}, provided $\cL$ is a densely defined closed operator, the adjoint of the projection $\P _\lambda$ is given by 
    \begin{align*}
        \P ^* _\lambda = \frac1{2 \pi i} \int _{\bar \Gamma} (z - \cL ^*) ^{-1} \d z,    
    \end{align*}
    with $\bar \Gamma$ being the reflection of $\Gamma$ about the real axis, positively oriented. Since the spectra of $\cL ^*$ and $\cL$ are symmetric about the real axis, $\bar \Gamma$ surrounds $\bar \lambda$ and nothing else in the spectrum of $\cL ^*$. $\P _\lambda ^*$ has the same rank as $\P _\lambda$, so its range is $\operatorname{Span} \{\phi ^*\}$ for some $\phi ^* \ne 0$. This shows $\P _\lambda ^* \phi ^* = \phi ^*$, thus $\phi ^*$ is an eigenvector of $\mf L ^*$, and $\bar \lambda$ is simple. Given this, we have 
    \begin{align*}
        ( \P _\lambda f, \phi ^* ) = ( f, \P _\lambda ^* \phi ^* ) = ( f, \phi ^* ). 
    \end{align*}
    Therefore, $L (f) ( \phi, \phi ^* ) = ( f, \phi ^* )$. Note that $( \phi, \phi ^* ) \ne 0$, because otherwise $( f, \phi ^* ) = 0$ for all $f$, which cannot be true since $\phi ^* \ne 0$ is an eigenvector. Hence $L (f)$ is the quotient of two inner products, and \eqref{eqn:riesz-equiv} is proven.

    Now we use $\cL = \cJ \HX$, then $\cL ^* = - \HX \cJ$. For an eigenvector $\phi$ with eigenvalue $\lambda \ne 0$, we have
    \begin{align}
        \label{eqn:eigenvector}
        \cL \phi = \cJ \HX \phi = \lambda \phi.
    \end{align}
    Applying $\HX$ to both sides yields that
    \begin{align*}
        \HX \cJ \HX \phi = \lambda \HX \phi \quad \leadsto  \quad -\cL ^* (\HX \phi) = \lambda (\HX \phi). 
    \end{align*}
    Hence, an eigenvector $\phi$ of $\cL$ with eigenvalue $\lambda$ is linked to $\HX \phi \in \cX$, an eigenvector of $\cL ^*$ with eigenvalue $-\lambda$. By \eqref{eqn:eigenvector}, $\phi = \frac1\lambda \cJ (\HX \phi)$, so this correspondence is one-to-one. Similarly, if $\phi ^-$ is an eigenvector of $\cL$ with eigenvalue $-\bar \lambda$, then $\phi ^* := \HX \phi ^-$ is an eigenvector of $\cL ^*$ with eigenvalue $\bar \lambda$.
\end{proof}

We recall the index theorem for Hamiltonian system in \cite{LinZeng2022}. Note that we have switched the letter $L$ and $H$ from their paper. Moreover, $\langle\cdot, \cdot \rangle$ refers to duality pair, and $T ^*: Y ^* \to X ^*$ refers to the dual operator of $T: X \to Y$.
\begin{theorem}[\cite{LinZeng2022}, theorem 2.3]
    \label{lem: lin-zeng}
    Let $\mf X$ be a real Hilbert space. Let $\mf L = \mf J \mf H: \mf X \to \mf X$ where $\mf H$ and $\mf J$ satisfy the following hypotheses.
    \begin{enumerate}[\upshape (H1)]
        \item $\mf J: \mf X ^* \supset D (\mf J) \to \mf X$ is anti self dual: $\mf J ^* = - \mf J$.
        \item $\mf H: \mf X \to \mf X ^*$ is bounded and symmetric: $\mf H ^* = \mf H$, and $\langle \mf H f, g \rangle$ is a bounded symmetric bilinear form in $\mf X$. Moreover, 
        \begin{align*}
            \mf X &= \mf X _- \oplus \ker \mf H \oplus \mf X _+, & n ^- (\mf H) := \dim \mf X _- < \infty.
        \end{align*}
        such that $\mf H$ is strictly negative in $\mf X _-$ and coercive in $\mf X _+$. 
        
        \item Let $i ^* _{\mf X _+ \oplus \mf X _-}: \mf X ^* \to (\mf X _- \oplus \mf X _+) ^*$ be the dual operator of the embedding $i _{\mf X _+ \oplus \mf X _-}$. Then 
        \begin{align*}
            \ker i ^* _{\mf X _+ \oplus \mf X _-} = \left\{
                f \in \mf X ^*: \langle f, u \rangle = 0, \forall u \in \mf X _- \oplus \mf X _+ 
            \right\} \subset D (\mf J).
        \end{align*}
    \end{enumerate}
    Then the following conclusions follow. 
    \begin{enumerate}[\upshape(i)]
        \item If $\lambda \in \sigma (\mf L)$, then $-\lambda, \bar \lambda, -\bar \lambda \in \sigma (\mf L)$.
        \item If $\lambda$ is an eigenvalue of $\mf L$, then so are $-\lambda, \bar \lambda, -\bar \lambda$. The dimension of each generalized eigenspace is identical.
        \item The following \textbf{index formula} holds: 
        \begin{align*}
            k _r + 2 k _c + 2 k _i ^{\le 0} + k _0 ^{\le 0} = n ^- (\mf H).
        \end{align*}
        Here $k _r$ is the sum of algebraic multiplicities of positive eigenvalues of $\mf L$, $k _c$ is is the sum of algebraic multiplicities of eigenvalues of $\mf L$ in the first quadrant,  $k _i ^{\le 0}$ is the number of non-positive directions of the quadratic form associated with $\mf H$ in the generalized eigenspaces of a purely imaginary eigenvalue $\lambda = i \mu$ with $\mu > 0$, and $k _0 ^{\le 0}$ is the number of non-positive directions of the quadratic form associated with $\mf H$ in the generalized eigenspaces of $0$ modulo $\ker \mf H$.
    \end{enumerate}
\end{theorem}

\subsection{Lyapunov exponent}
\label{sec:lyapunov}

To determine the eigenvalue $\lambda_\e$ of $\cLe$, the Lyapunov exponent of the perturbed velocity fields $\ue$ is required. We recall its definition as follows
\begin{align*}
    \Lambda _\e :=&  \limsup _{|t| \to \infty} \frac{1}{|t|} \log \left\| \nabla \Phie (t, \cdot) \right\|_{L^{\infty}},\quad\text {with } \begin{cases}\dot \Phi _\e (t, x, y) = \ue (\Phie (t, x, y)),\\ \Phie (0, x, y) = (x, y).\end{cases}
\end{align*}
Here $\Phie = (\Xe, \Ye)$ is the $\ue$-flow map and $\dot{(\cdot)},\ \nabla$ indicate the time and spatial derivatives, respectively.

Since the flow $\ue$ is perturbed from a shear which naturally has Lyapunov exponent $0$, we expect $\Lambda_\e$ to be small. This is the content of the next lemma. 
\begin{lemma}\label{lem:Lyapunov}
    Consider the flow $\ue$ in Lemma \ref{lem:fixed-point} with $\e < 1$. There exists a {constant $C$ depending on $B _4$ given in Proposition \ref{pro:Euler_Steady}}, such that $\Lambda _\e \le C \sqrt \e$. {Moreover, the following estimate holds:
    \begin{equation} 
        \label{eqn: nabla-Phi-exp}
        |\nabla \Phie (t, x, y) | \le 2 (|t| + 2) \expo{C \sqrt{\e }|t|}, \qquad \forall t \in \R , (x, y) \in \T ^2.
    \end{equation}}
\end{lemma}

\begin{proof}
It suffices to prove \eqref{eqn: nabla-Phi-exp} in the case when $t>0$, in the case when $t<0$, we simply change $u$ to $-u$.
    By chain rule, 
    \begin{align*}
        \nabla \dot \Phi _\e &= \nabla \ue (\Phie) \nabla \Phie = \nabla u _0 (\Phie) \nabla \Phie + \nabla (\ue- u_{0}) (\Phie) \nabla \Phie.
    \end{align*}
    Recall $\nabla u _0 (x, y) = \begin{pmatrix}
        0 & U' (y) \\ 0 & 0
    \end{pmatrix}$, and $\| \ue - u _0 \| _{C ^1} \le C \| \ye - \psi _0 \| _{H ^4} \le C B _4 \e$ by Proposition \ref{pro:Euler_Steady}. Therefore 
    \begin{align*}
        \nabla \dot \Phi _\e = \begin{pmatrix}
            \cO (\e) & U' (\Ye) + \cO (\e) \\
            \cO (\e) & \cO (\e) 
        \end{pmatrix} \nabla \Phie.
    \end{align*}
    As $U' (y) = 25 \cos 5y$ is bounded, we obtain
    \begin{align}
        \label{eqn: diffeq-Xe}
        \frac{\d}{\d t} |\nabla \Xe| &\le C \e |\nabla \Xe| + C |\nabla \Ye|, \quad 
        \frac{\d}{\d t} |\nabla \Ye| \le C \e (|\nabla \Xe| + |\nabla \Ye|).
    \end{align}
    Hence, 
    \begin{align*}
        \frac{\d}{\d t} \left(|\nabla \Xe| + \e ^{-1/2} |\nabla \Ye|\right) \le C \sqrt \e \left(|\nabla \Xe| + \e ^{-1/2} |\nabla \Ye|\right).
    \end{align*}
    By Gr\"onwall's inequality, 
    \begin{align*}
        |\nabla \Xe (t)| + \e ^{-1/2} |\nabla \Ye (t) | \le (|\nabla \Xe(0)| + \e ^{-1/2} |\nabla \Ye (0)|) \expo{C \sqrt \e t}.
    \end{align*}    
    Since $\nabla \Phie (0) = \I$, we obtain 
    \begin{align*}
        |\nabla \Ye (t)| \le (1 + \e ^{1/2}) \expo{C \sqrt \e t} \le 2 \expo{C \sqrt \e t}.
    \end{align*}
    Now solving \eqref{eqn: diffeq-Xe} with this upper bound on $\nabla \Ye$, we have 
    \begin{align*}
        \expo{-C \e t} |\nabla \Xe| (t) - |\nabla \Xe| (0) \le \int _0 ^t 2 \expo{C (\sqrt \e - \e) s} \d s \le 2 t \expo{C (\sqrt \e - \e) t}.
    \end{align*}
    Thus, we have $|\nabla \Xe| (t) \le 2 t \expo{C \sqrt \e t} + \expo{C \e t} \le 2 (t + 1) \expo{C \sqrt \e t}$.
 
    Combined with the bound on $\nabla \Ye$, This proves \eqref{eqn: nabla-Phi-exp}. Accordingly, we have
    \begin{align*}
        \lim _{t \to \infty} \frac1t \log \left\| \nabla \Phie (t, \cdot) \right \|_{L^{\infty}} \le C \sqrt \e.
    \end{align*}
    By definition, we have $\Lambda _\e \le C \e ^{1/2}$.
\end{proof}

The Lyapunov exponents controls higher order derivatives of the flow map $\Phi_t$. This is the content of the following lemma.

\begin{lemma}\label{lem:Hk-Lyapunov}
    For every $k \ge 1$, we have 
    \begin{align}
        \label{eqn: Hk-Lyapunov}
         \limsup _{t \to \infty} \frac{1}{|t|} \log \|\nabla ^k \Phie (t)\|_{L^{\infty}} \le k \Lambda _\e.
    \end{align}
    Moreover, there exist universal constant $C$ and $k$-dependent $C _k$ so that
    \begin{align}
        \label{eqn: nabla-k-Phi-exp}
        |\nabla ^k \Phie (t) |\le C _k (|t| + 2) ^{3 k-2} \expo{C k \sqrt \e |t|},\quad\forall t\in \R.
    \end{align}
\end{lemma}

\begin{proof}
The proof of \eqref{eqn: Hk-Lyapunov} was provided in \cite{ShvydkoyLatushkin2005}*{lemma 10}. Following their idea, one can obtain the bound \eqref{eqn: nabla-k-Phi-exp}. For the sake of completeness, we provide the proof here. 
Again it suffices to prove \eqref{eqn: Hk-Lyapunov} and \eqref{eqn: nabla-k-Phi-exp} in the case when $t>0$.
    The case $k = 1$ was given in \eqref{eqn: nabla-Phi-exp}. By chain rule (Fa\`a di Bruno formula), for any multi-index $\alpha \in \N ^2$,
    \begin{align*}
        \partial ^\alpha \dot \Phi _\e = \partial ^\alpha (\ue (\Phie)) = \nabla \ue (\Phie) \partial ^\alpha \Phie + \cR _\alpha,
    \end{align*}
    where the remainder has the form
    \begin{align*}
        \cR _\alpha = \sum _{\substack{2 \le i \le |\alpha| \\ \alpha _1 + \dots + \alpha _i = \alpha \\ |\alpha _j| > 0}} C _{\alpha _1, \dots, \alpha _i} \nabla ^i \ue (\Phie) \partial ^{\alpha _1} \Phie \dots \partial ^{\alpha _i} \Phie.
    \end{align*}
    By variation of parameters, for $|\alpha| \ge 2$ we have 
    \begin{align*}
        \partial ^\alpha \Phie (t, x, y) = \int _0 ^t \nabla \Phie (t - s, \Phie (s, x, y)) \cR _\alpha (s, x, y) \d s.
    \end{align*}
    If \eqref{eqn: nabla-k-Phi-exp} is true, then for $|\alpha| = k + 1$, as $\sum_{j=1}^{i} (3|\alpha_{j}|-2)=3(k+1)-2i\leq 3k-1$, by induction, we have
    \begin{align*}
        |\cR _\alpha (s, x, y)| &\le C _\alpha \| \ue \| _{C ^{k + 1}} (s + 2) ^{3 k-1} \expo{C(k+1) \sqrt \e s} \\
        & \le C _\alpha \| \ue \| _{C ^{k + 1}} (t + 2) ^{3k-1} \expo{C(k+1) \sqrt \e s}, \qquad 0 \le s \le t.
    \end{align*}
    Together with \eqref{eqn: nabla-Phi-exp} which bounds
    \begin{align*}
        |\nabla \Phie (t - s, \Phie (s, x, y))| \le 2 (t - s + 2) \expo{C \sqrt \e (t - s)} \le 2 (t + 2) \expo{C \sqrt \e (t - s)},
    \end{align*}
    we have 
    \begin{align*}
        |\partial ^\alpha \Phie (t, x, y)| \le t \cdot 2 C _\alpha \| \ue \| _{C ^{k + 1}} (t + 2) ^{3 k} \expo{C (k + 1) \sqrt \e t} \le C _{k + 1} (t + 2) ^{3 k + 1} \expo{C (k + 1) \sqrt \e t},
    \end{align*}
    where we used $\| \ue \| _{C ^{k + 1}} \le C \| \ye \| _{H ^{k + 4}}$ which is uniformly bounded in $\e$ by Proposition \ref{pro:Euler_Steady}. 
    As this is true for any $|\alpha| = k + 1$, we have shown \eqref{eqn: nabla-k-Phi-exp} for $k + 1$, and by induction it must hold for any $k \ge 1$.
\end{proof}

\subsection{Spectrum of the transport operator}
\label{sec:spec-transport}

The Lyapunov exponent bound yields an exponential control on the semigroup generated by the transport operator $\cTe = -\ue \cdot \nabla$ in $\cX ^k$. This is the content of the next lemma. 

\begin{lemma}
    \label{lem: transport semigroup}
    There exists a universal constants $C > 0$, such that for every $\gamma > 0$, $k \in \N$, there exists a constant $C _{k, \gamma} > 0$, such that for $\e \in [0, (\frac\gamma{2 C k}) ^2]$, 
    \begin{align*}
        \| \expo{t \cTe} \| _{\cX ^k \to \cX ^k} \le C _{k, \gamma} \expo{\gamma t}, \qquad \forall t \ge 0.
    \end{align*}
\end{lemma}

\begin{proof}
   
    Since $\cTe = -\ue \cdot \nabla$, the semigroup is exactly composition by the flow map
    \begin{align*}
        \expo{t \cTe} g = g \circ \Phie(-t), \qquad \forall g \in \cX.
    \end{align*}
    By chain rule, for any $|\alpha| \ge 1$ we have
    \begin{align*}
        \partial ^\alpha (g \circ \Phie)(-t) = \sum _{\substack{1 \le i \le |\alpha| \\ \alpha _1 + \dots + \alpha _i = \alpha \\ |\alpha _j| > 0}} C _{\alpha _1, \dots, \alpha _i} \nabla ^i g (\Phie)(-t) \partial ^{\alpha _1} \Phie(-t) \dots \partial ^{\alpha _i} \Phie(-t).
    \end{align*}
    Hence for $1 \le |\alpha| \le k$, we have the following by applying \eqref{eqn: nabla-k-Phi-exp}:
    \begin{align*}
        \| \partial ^\alpha (g \circ \Phie)(-t) \| _{L ^2} \le C _k \| g \| _{H ^k _0} (t + 2) ^{3 k-2} \expo{C k \sqrt \e t}.
    \end{align*}
    Note that $H ^k _0$ norm controls $\nabla ^i g$ with $1 \le i \le k$.
    Summing over $\alpha$ we get 
    \begin{align*}
        \| g \circ \Phie (-t)\| _{H ^k _0} \le C _k \| g \| _{H ^k _0} (t + 2) ^{3 k-2} \expo{C k \sqrt \e t}.
    \end{align*}
    With $\sqrt\e \le \frac\gamma{2 C k}$, we have 
    \begin{align}
        \| g \circ \Phie (-t)\| _{H ^k _0} \le C _k \| g \| _{H ^k _0} (t + 2) ^{3 k-2} \expo{\gamma t / 2}  \le C _{k, \gamma} \expo{\gamma t} \| g \| _{H ^k _0}.
    \end{align}
    Since every $\ue$ enjoys the symmetry enforced by $\cX$, the proof is completed.
\end{proof}

An immediate corollary is the resolvent estimate for the transport operator. Recall that for $z \notin \sigma (\cA)$ the resolvent of an operator $\cA$ is defined as 
\begin{align*}
    \cR _z (\cA) = (z - \cA) ^{-1}.
\end{align*}

\begin{corollary}
    \label{cor: resolvent}
    Under assumption of Lemma \ref{lem: transport semigroup}, the spectrum of $\cTe$ is contained in 
    \begin{align}
        \sigma (\cTe) \subset \{z \in \C: \lvert\Re z\rvert \le \gamma \}.    
    \end{align}
    In particular, for $z \in \C$ with $\lvert\Re z\rvert > \gamma$, we have 
    \begin{align*}
        \| \cR _z (\cTe) \| _{\cX ^k \to \cX ^k} \le \frac{C _{k, \gamma}}{\lvert\Re z\rvert - \gamma}.
    \end{align*}    
\end{corollary}

\begin{proof}
    By Lemma \ref{lem: transport semigroup}, the following integral is convergent for every $z$ with $\Re z > \gamma$:
    \begin{align*}
        \tilde {\cR} f := \int _0 ^\infty \expo{-t z} \expo{t \cTe} f \d t.
    \end{align*}
    One can verify $(z - \cTe) \tilde {\cR} f = \tilde {\cR} (z - \cTe) f = f$. So $\cR _z (\cTe) = \tilde {\cR}$, thus 
    \begin{align*}
        \| \cR _z (\cTe) f \| _{\cX ^k} \le \int _0 ^\infty \expo{-t \Re z} C _{k, \gamma} \expo{\gamma t} \| f \| _{\cX ^k} \d t = \frac{C _{k, \gamma}}{\Re z - \gamma} \| f \| _{\cX ^k}.
    \end{align*}
    This completes the proof for the resolvent bound when $\Re z > \gamma$. The case $\Re z < -\gamma$ is proven by flipping the direction of $\ue$.
\end{proof}

\subsection{Spectrum of the linearized Euler operator}
\label{sec:spec-euler}

In the previous section we have obtained that the spectrum $\sigma (\cTe)$ is inside a strip $\{\lvert\Re z\rvert \le \gamma\}$. We start by showing eigenvalues of $\cLe$ outside this strip is uniformly bounded.

From now on, let us fix $\gamma = \lambda _0 / 8$, and for $k\in \N$, we define 
\begin{align}
    \label{eqn: bar ek}
    \bar \e _k = \min \left\{\e^*, 1, \frac{\gamma^2}{16C^2(k+1)^2} \right\}.    
\end{align}

\begin{lemma}[Boundedness of unstable eigenvalue]
    \label{lem: eigen-bounded}
    Let $k \ge 0$. There exists $R > \lambda _0$, such that for all $\e \in [0, \bar \e _k]$, we have
    \begin{align*}
        \sigma (\cLe) \cap \{ z \in \C: \Re z > \gamma\} \subset B _R (0).
    \end{align*}
    Here $B _R (0) = \{ z \in \C: |z| < R\}$. Moreover, within 
    \begin{align}
        \label{eqn: K}
        K := \{z \in \C: \Re z \ge 2 \gamma, |z| \le R \}, 
    \end{align} 
    $\sigma (\cLe)$ contains only discrete eigenvalues with finite multiplicity. 
\end{lemma}

\begin{proof}
    Since $\cLe = \cTe + \cKe$ and $\cKe$ is compact,  $\sigma (\cLe) \setminus \sigma (\cTe)$  consists of only isolated eigenvalues with finite algebraic multiplicity. In particular, $\sigma (\cLe) \cap \{ z \in \C: \Re z > \gamma\}$ contain only eigenvalues with finite multiplicity.

    To show these eigenvalues are bounded, observe that for each $z \notin \sigma (\cTe)$, we have
    \begin{align}
        \label{eqn: magical-identity}
        z - \cLe = z - \cTe - \cKe = (z - \cTe) \left(
            \I - \cR _z (\cTe) \cKe
        \right).
    \end{align}
    Since $z - \cTe$ is invertible, we need to show $\cR _z (\cTe) \cKe$ is small when $z$ is large. This is can be seen by 
    \begin{align*}
        \cR _z (\cTe) (z - \cTe) = \I \quad \leadsto \quad z \cR _z (\cTe) = \cR _z (\cTe) \cTe + \I.
    \end{align*}
    Thus 
    \begin{align*}
        \| \cR _z (\cTe) \cKe \| _{\cX ^k \to \cX ^k} \le \frac1{|z|} \left( 
            \| \cR _z (\cTe)  \circ \cTe \circ \cKe \| _{\cX ^k \to \cX ^k} + \| \cKe \| _{\cX ^k \to \cX ^k}
        \right).
    \end{align*}
    Since $\cR _z (\cTe)$ is bounded by Corollary \ref{cor: resolvent}, $\cKe$ is uniformly bounded from $\cX ^k$ to $\cX ^{k + 1}$, $\cTe$ is uniformly bounded from $\cX ^{k + 1}$ to $\cX ^k$, we conclude that the right hand side is bounded, so 
    \begin{align}
        \label{eqn:1/z-bound}
        \| \cR _z (\cTe) \cKe \| _{\cX ^k \to \cX ^k} \le C |z| ^{-1} < \frac12, \qquad |z| > R
    \end{align}
    for some $R > \lambda _0$ depending on the bounds $B _k$'s, given in Proposition \ref{pro:Euler_Steady}. 
\end{proof}

We will show that $\sigma (\cLe)$ in $K$ contains only a single eigenvalue. To achieve this, let us first confirm $\cLe$ is invertible on $\partial K$, so a spectrum projection can be defined.

\begin{lemma}[Resolvent on $\partial K$]
    \label{lem:resolvent-L}
    Let $\Gamma = \partial K$ be the counterclockwise oriented boundary, for $K$ given in \eqref{eqn: K}. For $k \ge 0$, there exists a small positive $\bar \e _k' \le \bar \e _k$, such that for all $\e \in [0, \bar \e _k']$, it holds 
    \begin{align}
        \label{eqn: no-eigen-on-Gamma}
        \sigma (\cLe) \cap \Gamma = \varnothing.
    \end{align}
    Moreover, $\cR _z (\cLe)$ is bounded in $\cX ^k$ operator norm, uniformly for $z \in \Gamma$ and $\e \in [0, \bar \e _k']$.
\end{lemma}

\begin{proof}
    First, we prove \eqref{eqn: no-eigen-on-Gamma} by contradiction. If no $\bar \e _k'$ exists, then we have a sequence $\e _j \to 0$ and a sequence of eigenvalue $\lambda _j \in \sigma (\cL _{\e _j}) \cap \Gamma$ with eigenvector $\phi _j$, normalized in $\cX ^k$. Compactness of $\Gamma$ ensures $\lambda _j$ converges to $\lambda _* \in \Gamma$ up to a subsequence. For each $\phi _j$, \eqref{eqn: magical-identity} implies
    \begin{align}
        \label{eqn: phi-j}
        (\lambda _j - \cL _{\e _j}) \phi _j = 0 \quad \leadsto \quad \phi _j =  \cR _{\lambda _j} (\cT _{\e _j}) \cK _{\e _j} \phi _j.
    \end{align}
    $\cKe$ is bounded $\cX ^k$ to $\cX ^{k + 1}$, and $\cR _z (\cTe)$ is bounded from $\cX ^{k + 1}$ to $\cX ^{k + 1}$ by \eqref{cor: resolvent}. Therefore, $\phi _j$ is uniformly bounded in $\cX ^{k + 1}$. By compact embedding $H ^{k + 1} \hookrightarrow H ^k$, a subsequence of $\phi _j$ converges to $\phi _* \ne 0$ in $\cX ^k$. 
    
    To confirm $\phi _*$ is indeed an eigenfunction of $\lambda _*$, we take the subsequential limit $j \to \infty$ in \eqref{eqn: phi-j}. The continuity of $\cKe$ in $\e$ is straightforward. Because $\nabla \om _{\e _j} \to \nabla \om _0$ in any $\cX ^k$ norm, we have
    \begin{align*}
        \| \cK _{\e _j} - \cK _0 \| _{\cX ^k \to \cX ^{k + 1}} \to 0 \qquad \text{ as } j \to \infty.
    \end{align*}
    Next we prove continuity of $\cR _z (\cT _\e)$ in $z$ and $\e$. By the first resolvent identity:
    \begin{align}
        \label{eqn: first-resolvent}
        \cR _{\lambda _j} (\cT _{\e _j}) - \cR _{\lambda _*} (\cT _{\e _j}) = (-\lambda _j +\lambda _*) \cR _{\lambda _j} (\cT _{\e _j}) \cR _{\lambda _*} (\cT _{\e _j}) \to 0
    \end{align}
    in $\cX ^k$ operator norm as $j \to \infty$, as both resolvent are bounded by Corollary \ref{cor: resolvent} and $\lambda _j \to \lambda _*$. By the second resolvent identity:
    \begin{align*}
        \cR _{\lambda _*} (\cT _{\e _j}) - \cR _{\lambda _*} (\cT _0) = \cR _{\lambda _*} (\cT _{\e _j}) \circ (-\cT _0 + \cT _{\e _j}) \circ \cR _{\lambda _*} (\cT _0).
    \end{align*}
    Because $\nabla \om _{\e _j} \to \nabla \om _0$ in any $\cX ^k$ norm, we have
    \begin{align}
        \label{eqn: continuity-cT}
        \| \cT _{\e _j} - \cT _0 \| _{\cX ^{k + 1} \to \cX ^k} \to 0 \qquad \text{ as } j \to \infty.
    \end{align}
    Using again resolvent $\cR _z (\cTe)$ are bounded in $\cX ^k$ and $\cX ^{k + 1}$, we conclude
    \begin{align}
        \label{eqn: convergence-RT}
        \| \cR _{\lambda _*} (\cT _{\e _j}) - \cR _{\lambda _*} (\cT _0) \| _{\cX ^{k + 1} \to \cX ^k} \to 0 \qquad \text{ as } j \to \infty.
    \end{align}
    In summary, we have
    \begin{align*}
        \| \phi _j - \phi _* \| _{\cX ^k} &\to 0, &
        \| \cK _{\e _j} - \cK _0 \| _{\cX ^k \to \cX ^{k + 1}} &\to 0, &
        \| \cR _{\lambda _j} (\cT _{\e _j}) - \cR _{\lambda _*} (\cT _0) \| _{\cX ^{k + 1} \to \cX ^k} &\to 0.
    \end{align*}
    Hence we can pass $j \to \infty$ in \eqref{eqn: phi-j}, and conclude 
    \begin{align*}
        \phi _* = \cR _{\lambda _*} (\cT _0) \cK _0 \phi _*.
    \end{align*}
    
    Using \eqref{eqn: magical-identity} again, we showed $(\lambda _*, \phi _*)$ is an eigenpair of $\cL _0$, but this is impossible because $\sigma (\cL _0) \subset \{ \lambda _0, -\lambda _0 \} \cup i \R$ by Proposition \ref{pro:Kolmogorov}, which is disjoint from $\Gamma$. This proves \eqref{eqn: no-eigen-on-Gamma} must hold for sufficiently small $\e$. 

    Uniform boundedness of $\cR _z (\cLe)$ in $\cX ^k$ norm is proven similar. Suppose not, then there exist $\e _j \to 0$, $z _j \in \Gamma$ and $f _j \in \cX ^k$, such that 
    \begin{align}
        \label{eqn: compactness-fk}
        \| f _j \| _{\cX ^k} &= 1, &
        \| (z _j - \cL _{\e _j}) f _j \| _{\cX ^k} & \le \frac1j 
        \quad \leadsto \quad 
        \lim _{j \to \infty} \| (\I -\cR _{z _j} (\cT _{\e _j}) \cK _{\e _j}) f _j \| _{\cX ^k} = 0.
    \end{align}
    Here we used \eqref{eqn: magical-identity} to show the convergence. Same as before, up to a subsequence, $z _j \to z _* \in \Gamma$, and $\cR _{z _j} (\cT _{\e _j}) \cK _{\e _j} f _j$ is uniformly bounded in $\cX ^{k + 1}$, so a subsequence converges in $\cX ^k$. By \eqref{eqn: compactness-fk}, we know $f _j \to f _*$ in $\cX ^k$ for some $\| f _* \| _{\cX ^k} = 1$, and via the same limiting process we have 
    \begin{align*}
        (\I - \cR _{z _*} (\cT _0) \cK _0) f _* = 0.
    \end{align*}
    By \eqref{eqn: magical-identity}, $z _* - \cL _0$ is not invertible, contradicting \eqref{eqn: no-eigen-on-Gamma}.
\end{proof}

We are ready to characterize the spectrum of $\cLe$, and prove convergence of the most unstable eigenvalue and eigenfunction.

\begin{proposition}[Persistence of eigenpair]
    \label{prop: eigenpair}
    For $\e \in [0, \bar \e _k']$ with $\bar \e _k' > 0$ given in Lemma \ref{lem:resolvent-L}, we have
    \begin{align*}
        \sigma (\cLe) \cap \{ z \in \C: \Re z \ge 2 \gamma\} = \{\lambda _\e\}.
    \end{align*}
    $\lambda _\e > 2 \gamma$ is a real simple eigenvalue that changes continuously in $\e$, and $\lambda _\e \to \lambda _0$ as $\e \to 0$. The associated eigenfunction $\fe$ also changes continuously in $\e$ and $\fe \to \phi _0$ in $\cX ^k$. Moreover, the Riesz projection $\P _{\lambda _\e}$ of $\cLe$ converges to Riesz projection $\P _{\lambda _0}$ of $\cL _0$ in $\cX ^k$ operator norm.
\end{proposition}

\begin{proof}

    Since $\Gamma$ contains no eigenvalue of $\cLe$ for sufficiently small $\e$ and $\cR _z (\cLe)$ is bounded for $z \in \Gamma$, we are able to define the spectrum projection 
    \begin{align}
        \label{eqn: riesz-projection}
        \P _\e f := \frac1{2 \pi i} \oint _\Gamma \cR _z (\cLe) f \d z.
    \end{align}
    Next we show convergence $\P _\e \to \P _0$. Fix $z \in \Gamma$. By the second resolvent identity implies 
    \begin{align}
        \label{eqn: second-resolvent-id}
        \cR _z (\cLe) - \cR _z (\cL _0) = \cR _z (\cLe) \circ (-\cL _0 + \cLe) \circ \cR _z (\cL _0).
    \end{align}
    As $\| \cLe - \cL _0 \| _{\cX ^{k + 1} \to \cX ^k} \to 0$, we have
    \begin{align*}
        \| \cR _z (\cLe) - \cR _z (\cL _0) \| _{\cX ^{k + 1} \to \cX ^k} \to 0 \qquad \text{ as } \e \to 0, \qquad \text{ uniformly for } z \in \Gamma.
    \end{align*}
    Together with boundedness of resolvent in Lemma \ref{lem:resolvent-L}, we get $\delta_\e:=\|\P_\e-\P_0\|_{\cX^{k+1}\to\cX^k}\to0$.
  Moreover,  the identity
\[
 \cR_z(\cLe)=\cR_z(\cTe)+\cR_z(\cLe)\cKe\cR_z(\cTe)
\]
and analyticity of the transport resolvent inside $\Gamma$ give
\[
 \P_\e=\frac{1}{2\pi i}\oint_\Gamma
 \cR_z(\cLe)\cKe\cR_z(\cTe)\d z,
 \qquad \|\P_\e\|_{\cX^{k-1}\to\cX^k}\le M.
\]
Now let $F_N$ be the Fourier cutoff to frequencies of length at most $N$, there exists $C>0$ such that 
\[
 \|F_N\|_{H^k\to H^{k+1}}\le C(1+N),\qquad
 \|I-F_N\|_{H^k\to H^{k-1}}\le C/(1+N),
\]
here $\cX^{-1}$ denote the Hilbert space endowed with $H^{-1}$ norm.
Thus
\[
 \|\P_\e-\P_0\|_{\cX^k\to\cX^k}
 \le C\delta_\e(1+N)+2CM/(1+N).
\]
In particular, by letting $N$ large and $\e$ small, we have $\lim_{\e\rightarrow 0} \|\P_\e-\P_0\|_{\cX^{k}\to\cX^k}=0$ so that $\Pe$ and $\P _0$ have the same rank. Since $\lambda _0$ is the unique eigenvalue of $\cL _0$ in $K$, we know $\P _0$ has rank 1, so $\P _\e$ also has rank 1, implying $\cLe$ has a unique eigenvalue $\lambda _\e$ in $K$ which is simple. This shows $\P _\e$ is exactly the Riesz projection of $(\lambda _\e, \fe)$ for operator $\cLe$. Since $\Pe \to \P _0$, the normalized eigenfunctions converge $\fe \to \phi _0$ in $\cX ^k$, and together with $\| \cLe - \cL _0 \| _{\cX ^{k + 1} \to \cX ^k} \to 0$ we conclude $\lambda _\e \to \lambda _0$. Finally, because $(\bar \lambda _\e, \bar \phi _\e)$ is another eigenpair of $\cLe$ in $K$, we know $\lambda _\e$ must be real, and \eqref{eqn: no-eigen-on-Gamma} shows $\lambda _\e > 2 \gamma$. 

    As we have shown that $\lambda _\e$ is the only point of $\sigma (\cLe)$ in $K$, we can show $\lambda _\e$ changes in $\e$ continuously in the same way as we show $\lambda _\e \to \lambda _0$ as $\e \to 0$. Continuity of eigenfunction in $\e$ follows the continuity of projection.
\end{proof}

Finally we prove the projected semigroup bound. Before that let us recall the classical Gearhart--Pr\"uss--Huang--Greiner theorem.

\begin{proposition}[Theorem 1.3.1 in \cite{LiuZheng1999}.]
\label{thm:GPHG}
Let $\mf X$ be a  Hilbert space, and let
\begin{align*}
    \mf A: \mf X \supset D(\mf A) \to \mf X
\end{align*}
be the generator of a strongly continuous semigroup $\{\expo{t \mf A}\} _{t \ge 0}$.
Then the following statements are equivalent:
\begin{enumerate}
\item There exist constants $M \ge 1$ and $\eta > 0$ such that
\begin{align*}
    \|\expo{t \mf A} \|_{\mf X \to \mf X}
    \le M \expo{-\eta t}, 
    \qquad t\ge 0.
\end{align*}

\item The open right half-plane belongs to the resolvent set of $\mf A$, i.e.
\begin{align*}
    \C_+ := \{z\in\C:\operatorname{Re}z>0\} \subset \rho(\mf A),
\end{align*}
and we have the following uniform bound on the resolvent:
\begin{align*}
    \sup_{\Re z>0}
    \|(z \I - \mf A) ^{-1} \|_{\mf X \to \mf X}
    <\infty.
\end{align*}
\end{enumerate}
\end{proposition}

\begin{lemma}[Bounds on the semigroup]
    \label{lem: semigroup}
    Denote $\Q _\e = \I - \P _\e$. Then there exists $M > 0$ depending on $\ue$ and $m$ such that 
    \begin{align*}
        \left\| 
            \expo{t \cLe} \Q _\e
        \right\| _{\cX ^k \to \cX ^k} \le M^2 \expo{2 \gamma t}.
    \end{align*}
\end{lemma}

\begin{proof}
    As the spectral projection $\Q _\e$ removes the eigenvalue $\lambda _\e$ from the spectrum of $\cL _\e$, we know that $\sigma (\cLe \circ \Q _\e) \subset \{ z \in \C: \Re z \le 2 \gamma \}$. We will show $\cR _z (\cLe \circ \Q _\e)$ is bounded in $\X _\e ^k := \Qe \cX ^k$ operator norm uniformly in $\{ \Re z > 2 \gamma \}$ for $\e \in [0, \bar \e _k']$. It is bounded for $|z| > R$ by Lemma \ref{lem: eigen-bounded}, so we only need to show $\cR _z (\cLe \circ \Q _\e)$ is bounded uniformly for $z \in K$.

    Suppose otherwise, then we have a sequence of $f _j \in \X _\e ^k$ with $\| f _j \| _{\cX ^k} = 1$, $\e _j \in [0, \bar \e _k']$, $z _j \in K$, that 
    \begin{align*}
        \| (z _j - \cL _{\e _j} \circ \Q _{\e _j}) f _j \| _{\cX ^k} \to 0 \quad \leadsto \quad \| (z _j - \cL _{\e _j}) f _j \| _{\cX ^k} \to 0.
    \end{align*}
    In view of \eqref{eqn: magical-identity} and resolvent bound Corollary \ref{cor: resolvent}, we have 
    \begin{align*}
        \| (\I - \cR _{z _j} (\cT _{\e _j}) \cK _{\e _j}) f _j \| _{\cX ^k} \to 0 \qquad \text{ as } j \to \infty.
    \end{align*}
    Same as the scenario in \eqref{eqn: compactness-fk}, we get $f _j \to f _*$ in $\cX ^k$ up to a subsequence, and 
    \begin{align*}
        z _j &\to z _*, &
        \e _j &\to \e _*, &
        f _j &\to f _* \text{ in } \cX ^k, & 
        \quad \leadsto \quad 
        (\I - \cR _{z _*} (\cT _{\e _*}) \cK _{\e _*}) f _* &= 0 \quad \leadsto \quad \cL _{\e _*} f _* = z _* f _*.
    \end{align*}
    Since $\Re z _* \ge 2 \gamma > 0$, we know $f _*$ is the eigenfunction of eigenvalue $z _* = \lambda _{\e _*}$, thus $\P _{\e _*} f _* = f _*$. However, $f _j \in \X ^k _{\e _j}$ implies $\Q _{\e _j} f _j = f _j$ and $\P _{\e _j} f _j = 0$. Passing $j \to \infty$ yields $\P _{\e _*} f _* = f _* = 0$, contradicting with $f _* = \lim _{j \to \infty} f _j$ in $H ^k$ where $\| f _j \| _{H ^k} = 1$.

    We have shown that $\cR _z (\cLe \circ \Q _\e)$ is bounded in $\X ^k _\e$, uniformly for $\{\Re z > 2 \gamma\}$ and $\e \in [0, \bar \e _k']$. By applying Proposition \ref{thm:GPHG} to $\mf A = \cLe \circ \Q _\e - 2 \gamma \I$, we have 
    \begin{align*}
        \lVert \expo{t \cLe \Q _\e} \rVert _{\X ^k _\e \to \X ^k _\e} \le M \expo{(2 \gamma - \eta) t} \le M \expo{2 \gamma t}.
    \end{align*}
    Now take $M\geq \| \Qe \|_{\X ^k\to \X ^k}$,
    we have  $\lVert \expo{t \cLe} \Qe \rVert _{\cX ^k \to \cX ^k} \leq \lVert \expo{t \cLe \Q _\e} \rVert _{\X ^k _\e \to \X ^k _\e} \| \Qe \|_{\X ^k\to \X ^k}\leq M^2 e^{2\gamma t}$.
\end{proof}

\begin{proof}[Proof of Proposition \ref{pro:Euler_spectral}, Item 1.~and 3.]
    These two statements have been proven by Proposition \ref{prop: eigenpair} and Lemma \ref{lem: semigroup}, respectively, by choosing $\e ^{**} \le \min _{0 \le k \le 2 m} \bar \e _k '$ sufficiently small such that for all $0 \le \e \le \e ^{**} $ we have  $2 \gamma := \frac14 \lambda _0 < \frac12 \lambda _\e$.
\end{proof}

\section{Continuous fraction}
\label{sec:continuous fraction}

This section provides a self-contained proof on the convergence of continued fraction and its properties.

\begin{lemma}
    \label{prop:friedlander1998}
    Let $\{\beta_n\}_{n=1}^\infty$ be a sequence of positive numbers such that $\lim _{n\to\infty} \beta _n = c > 0$. For all $\lambda>0$, define 
    \begin{align*}
        F_n(\lambda)
        &:= [\lambda \beta_1, \lambda \beta_2, \dots, \lambda \beta_n] = \frac1{
            \lambda \beta_1 + \dfrac1{
                \lambda \beta_2 + \dfrac1{
                    \begin{matrix}
                        \ddots & \\
                        &\displaystyle\frac1{\lambda \beta_n}
                    \end{matrix}
                }
            }
        }.
    \end{align*}
    Then for all $\lambda>0$, $\lim_{n\to\infty} F_n (\lambda) = F(\lambda)$, where $F$ is a positive continuous function in $\R^+$, with $\lim_{\lambda\to0^+}F(\lambda)=1$, and $\lim_{\lambda\to+\infty}F(\lambda)=0$. Moreover, for every fixed $\lambda>0$, each shifted infinite fraction is well-defined, and
    \begin{align}
        \label{fried_con1}
        &\lim_{n\to\infty} [\lambda\beta_{n+1},\lambda\beta_{n+2},\dots] 
        =\sqrt{1+\left(\frac{\lambda c}{2}\right)^2} -\frac{\lambda c}{2}.
    \end{align}
\end{lemma}

The proof of the convergence is given in the appendix of \cite{friedlander1998instability}.
For completeness, we give a self-contained simple proof here.
We also prove the additional conclusion \eqref{fried_con1}.

\begin{proof}
    Note that a continued fraction $[a_1,a_2,\dots,a_n]$ with positive $a_i$ is monotonically increasing in $a_2,a_4,\dots$ and decreasing in $a_1,a_3,\dots$, so
    \begin{align*}
        0<F_2<F_4<F_6<\dots<F_5<F_3<F_1,
    \end{align*}
    hence $F_n$ with odd $n$ and even $n$ converge respectively. Denote
    \begin{align*}
        F_e(\lambda)&=\lim_{n\to\infty}F_{2n}(\lambda),&
        F_o(\lambda)&=\lim_{n\to\infty}F_{2n+1}(\lambda).
    \end{align*}
    To prove $F$ is well-defined, we need to show $F_e=F_o$.

    Given $0 < \e <c$, there exists a large even $N>0$ such that $0<c-\e<\beta_n<c+\e$ for all $n>N$. Let us introduce remainder bounds
    \begin{align*}
        R^-(\lambda)
        &=[\lambda(c+\e),\lambda(c-\e),\lambda(c+\e),
        \lambda(c-\e),\dots],\\
        R^+(\lambda)
        &=[\lambda(c-\e),\lambda(c+\e),\lambda(c-\e),
        \lambda(c+\e),\dots].
    \end{align*}
    They are continuous functions with the following explicit expressions:
    \begin{equation}\label{eq:R+-}
        \begin{aligned}
            R^\pm(\lambda)
            &:=\frac12\left(\sqrt{1+\frac4{\lambda^2(c^2-\e^2)}}-1\right)
                \cdot\lambda(c\pm\e)\\
            &=\frac12\left(\sqrt{\lambda^2(c\pm\e)^2+
            \frac{4(c\pm\e)^2}{c^2-\e^2}}-\lambda(c\pm\e)\right).
        \end{aligned}
    \end{equation}
    The computation is deferred to the end of the proof, which also proves convergence of these two continued fractions, using only the odd/even monotonicity above.
    For $n>N$, let us write
    \begin{align*}
        F_n(\lambda)=f_N(\lambda,z)
        &:=[\lambda\beta_1,\lambda\beta_2,\dots,
        \lambda\beta_N+z],\\
        \text{with }z&=[\lambda\beta_{N+1},\lambda\beta_{N+2},
        \dots,\lambda\beta_n].
    \end{align*}
    Here $z$ is the value of the remaining fraction. For even $N$, $f _N$ is increasing
    in $z$.
    By monotonicity,
    \begin{align}
        \label{eq:comparison}
        R^-(\lambda)
        &\le\liminf_{n\to\infty}
        [\lambda\beta_{N+1},\lambda\beta_{N+2},\dots,\lambda\beta_n]
        \notag\\
        &\le\limsup_{n\to\infty}
        [\lambda\beta_{N+1},\lambda\beta_{N+2},\dots,\lambda\beta_n]
        \le R^+(\lambda).
    \end{align}
    As a consequence,
    \begin{align}
        \label{eqn:squeeze-FeFo}
        f_N(\lambda,R^-(\lambda))
        \le F_e(\lambda)
        \le F_o(\lambda)
        \le f_N(\lambda,R^+(\lambda)).
    \end{align}
    Moreover, $f_N$ is $1$-Lipschitz in $z$. We begin with the two-entry calculation:
    \begin{align*}
        \frac{\d}{\d z}[a_1,z]
        &=\frac{\d}{\d z}\frac1{a_1+1/z}
        =\left(
        \frac1{a_1+1/z}
        \right)^2 \left(\frac1z\right)^2
        =\frac1{(a_1z+1)^2}<1.
    \end{align*}
    Similarly, a continued fraction is also 1-Lipschitz in the fourth entry:
    \begin{align*}
        \frac{\d}{\d z}[a_1,a_2,a_3,z]
        &=\frac{\d}{\d z}[a_1,a_2+[a_3,z]] = \frac{\d}{\d a _2} [a_1,a_2+[a_3,z]] \times \frac{\d}{\d z}[a_3,z] < 1 \times 1 = 1.
    \end{align*}
    Continuing this decomposition, we have 
    \begin{align*}
        \frac{\d}{\d z}[a_1,a_2,\dots,a_{N - 1},z] = \frac{\d}{\d a _2}[a_1,a_2,\dots,a_{N - 1},z] \times \frac{\d}{\d a _4}[a_3,a_4,\dots,a_{N - 1},z] \times \dots \times \frac{\d}{\d z}[a_{N - 1},z],
    \end{align*}
    which is smaller than 1. Hence a continued fraction is 1-Lipschitz in all the even entries. In particular, as $N$ is even, $f _N (\lambda, z)$ is 1-Lipschitz in $z$. Therefore, we get
    \begin{align*}
        F_o(\lambda)-F_e(\lambda)
        &\le
        f_N(\lambda,R^+(\lambda))-f_N(\lambda,R^-(\lambda))\\
        &\le R^+(\lambda)-R^-(\lambda)
        =\left(\sqrt{1+\frac4{\lambda^2(c^2-\e^2)}}-1\right)
          \cdot\lambda\e,
    \end{align*}
    as can be seen from the closed-form expression of $R ^\pm$. Since $\e$ can be taken arbitrarily small, we have $F_o(\lambda)=F_e(\lambda)$, so the limit $F(\lambda)$ exists. It is positive because $F(\lambda)\ge F_2(\lambda)>0$.

    Continuity of $F(\lambda)$ can be proven similarly. Fix $\e$, its corresponding even $N$, and $\lambda_1>0$. Using \eqref{eqn:squeeze-FeFo} and Lipschitz property of $f _N$ in $z$, we have for all $\lambda _2 > 0$ that
    \begin{align*}
        F(\lambda_2)-F(\lambda_1)
        &=F_o(\lambda_2)-F_e(\lambda_1)\\
        &\le f_N(\lambda_2,R^+(\lambda_2))
        -f_N(\lambda_1,R^-(\lambda_1))\\
        &= f_N(\lambda_2,R^+(\lambda_2))
        -f_N(\lambda_2,R^-(\lambda_1))
        +f_N(\lambda_2,R^-(\lambda_1))
        -f_N(\lambda_1,R^-(\lambda_1))\\
        &\le R^+(\lambda_2)-R^-(\lambda_1)
        +f_N(\lambda_2,R^-(\lambda_1))
        -f_N(\lambda_1,R^-(\lambda_1)).
    \end{align*}
    Now we take $\lambda_2\to\lambda_1$. Since both $R^\pm$ and $f _N$ are continuous in $\lambda$, we have
    \begin{align*}
        \limsup_{\lambda_2\to\lambda_1}
        \bigl(F(\lambda_2)-F(\lambda_1)\bigr)
        &\le R^+(\lambda_1)-R^-(\lambda_1)=\left(\sqrt{1+\frac4{\lambda_1^2(c^2-\e^2)}}-1\right)
          \cdot\lambda_1\e,
    \end{align*}
    which can be made arbitrarily small by taking $\e\to0$. With the same argument with liminf, we obtain
    \begin{align*}
        \liminf_{\lambda_2\to\lambda_1}
        \bigl(F(\lambda_2)-F(\lambda_1)\bigr)
        \ge-\bigl(R^+(\lambda_1)-R^-(\lambda_1)\bigr).
    \end{align*}
    Letting $\e\downarrow0$ in both bounds,  we have proven the continuity of $F$ in $\R_+$.

    To see the limit of $F (\lambda)$ as $\lambda\to0^+$, note that the following limit holds:
    \begin{align*}
        \lim_{\lambda\to0^+} R^\pm(\lambda)
        &=\lim_{\lambda\to0^+}\frac12\left(
        \sqrt{\lambda^2(c\pm\e)^2+
        \frac{4(c\pm\e)^2}{c^2-\e^2}}-\lambda(c\pm\e)\right)=\frac{c\pm\e}{\sqrt{c^2-\e^2}}.
    \end{align*}
    Moreover, for each fixed even $N$, we have
    \begin{align*}
        \lim_{\lambda\to0^+}f_N(\lambda,z)=f_N(0,z)=z.
    \end{align*}
    Hence we can pass $\lambda \to 0$ in \eqref{eqn:squeeze-FeFo}, and obtain
    \begin{align*}
        \frac{c-\e}{\sqrt{c^2-\e^2}}
        &\le \liminf _{\lambda\to0^+} F(\lambda)
        \le \limsup_{\lambda\to0^+} F(\lambda)
        \le\frac{c+\e}{\sqrt{c^2-\e^2}}.
    \end{align*}
    Taking $\e\to0$ yields $\lim_{\lambda\to0^+}F(\lambda)=1$. As for $\lambda\to+\infty$ we can simply use that $F_n(\lambda)\le F_1(\lambda)=1/(\lambda\beta_1)\to0$.

    Finally, we prove \eqref{fried_con1}. For each $n \ge 0$, the shifted sequence $\{\beta_{n + j}\}_{j\ge1}$ satisfies the same hypotheses. The convergence argument above therefore defines
    \begin{align*}
        \cF _n (\lambda) :=[\lambda\beta_{n+1}, \lambda\beta_{n+2}, \dots].
    \end{align*}
    Fix $\lambda>0$ and $0<\e<c$. For all $n \ge N$, every coefficient in this tail lies between $c-\e$ and $c+\e$. The same comparison as in \eqref{eq:comparison} gives
    \begin{align*}
        R^-(\lambda)\le \cF _n (\lambda)\le R^+(\lambda),
        \qquad n \ge N.
    \end{align*}
    Thus
    \[
        R^-(\lambda)
        \le\liminf_{n\to\infty} \cF _n (\lambda)
        \le\limsup_{n\to\infty} \cF _n (\lambda)
        \le R^+(\lambda).
    \]
    As $\e\downarrow0$, both $R ^\pm$ tend to $\frac12 (\sqrt{\lambda^2c^2+4}-\lambda c) =\sqrt{1+\left(\frac{\lambda c}{2}\right)^2} -\frac{\lambda c}{2}.$ This proves \eqref{fried_con1}.
\end{proof}

\begin{proof}[Proof of \cref{eq:R+-}]
    Define
    \begin{align*}
        R_n^-(\lambda)
        &:=[\lambda(c+\e),\lambda(c-\e),\lambda(c+\e),\dots,
        \lambda(c+(-1)^{n+1}\e)],\\
        R_n^+(\lambda)
        &:=[\lambda(c-\e),\lambda(c+\e),\lambda(c-\e),\dots,
        \lambda(c+(-1)^n\e)].
    \end{align*}
    We can similarly define $R_o^\pm(\lambda)$ and $R_e^\pm(\lambda)$ as the limit of odd and even terms. They satisfy the following relation:
    \begin{align*}
        R_o^-(\lambda)&=\frac1{\lambda(c+\e)+R_e^+(\lambda)},&
        R_e^+(\lambda)&=\frac1{\lambda(c-\e)+R_o^-(\lambda)}.
    \end{align*}
    Simplifying these equations, one first observes that
    \begin{align*}
        \lambda(c+\e)R_o^-(\lambda) =1-R_o^-(\lambda)R_e^+(\lambda) =\lambda(c-\e)R_e^+(\lambda)=:x.
    \end{align*}
    Multiplying $1-x$ by $\lambda^2(c^2-\e^2)$, we have
    \begin{align*}
        \lambda^2(c^2-\e^2)(1-x)
        &=\lambda^2(c^2-\e^2)R_o^-(\lambda)R_e^+(\lambda)=\lambda(c+\e)R_o^-(\lambda)
        \cdot\lambda(c-\e)R_e^+(\lambda)=x^2.
    \end{align*}
    The quadratic equation has only one positive root
    \begin{align*}
        x=\frac{\sqrt{b^2+4b}-b}{2} =\frac12\left(\sqrt{1+4/b}-1\right)b, \qquad b=\lambda^2(c^2-\e^2).
    \end{align*}
    Hence we can solve for $R_o^-$ and $R_e^+$ as
    \begin{align*}
        R_o^-(\lambda)
        &=\frac{x}{\lambda(c+\e)}
        =\frac12\left(\sqrt{1+\frac4{\lambda^2(c^2-\e^2)}}-1\right)
        \cdot\lambda(c-\e),\\
        R_e^+(\lambda)
        &=\frac{x}{\lambda(c-\e)}
        =\frac12\left(\sqrt{1+\frac4{\lambda^2(c^2-\e^2)}}-1\right)
        \cdot\lambda(c+\e).
    \end{align*}
    $(R_o^-,R_e^+)$ solves the same system as $(R_e^-,R_o^+)$, and the positive solution is unique, so $R_o^\pm=R_e^\pm$, therefore we can omit the subscript and conclude \eqref{eq:R+-}.
\end{proof}

\begin{remark}
    \label{rmk:Fried_con}
    In the calculation for sinusoidal profile in \cref{sec:Unstable_e_function}, the positive coefficients are $\beta_{-6}, \beta_{-11}, \dots$, with $\beta_{-1-5n} \to \frac2{15}$ as $n \to \infty$. Set $\tilde\beta_n:=\beta_{-1-5n}$, $j\ge1$. For $n=-1-5\tilde n$, we have
    \begin{align*}
        [\lambda \beta_{n-5}, \lambda \beta_{n-10}, \dots]
        = [\lambda \tilde \beta_{\tilde n+1}, 
          \lambda \tilde \beta_{\tilde n+2},\dots].
    \end{align*}
    Applying \eqref{fried_con1} with $c=\frac2{15}$ gives the desired formula
    \begin{align*}
        \lim_{\substack{n \to -\infty \\ n\equiv4\;({\rm mod}\;5)}}
        [\lambda \beta_{n-5}, \lambda \beta_{n-10},\dots]
        =\sqrt{1+\left(\frac{\lambda}{15}\right)^2}
         -\frac{\lambda}{15}, \qquad \lambda > 0.
    \end{align*}
\end{remark}
\scriptsize
\bibliographystyle{plain}
\bibliography{reference/citation}

\begin{bibdiv}
\begin{biblist}

\bib{Acheson1990}{book}{
      author={Acheson, D.~J.},
       title={Elementary fluid dynamics},
   publisher={Oxford University Press},
        date={1990},
}

\bib{AlbrittonEtAl2024Vishik}{book}{
      author={Albritton, Dallas},
      author={Bru{\'e}, Elia},
      author={Colombo, Maria},
      author={De~Lellis, Camillo},
      author={Giri, Vikram},
      author={Janisch, Maximilian},
      author={Kwon, Hyunju},
       title={Instability and non-uniqueness for the 2{D} {E}uler equations,
  after {M}. {V}ishik},
      series={Annals of Mathematics Studies},
   publisher={Princeton University Press},
     address={Princeton, NJ},
        date={2024},
      volume={219},
        ISBN={978-0-691-25752-5},
}

\bib{bardos2002stable}{article}{
      author={Bardos, C.},
      author={Guo, Y.},
      author={Strauss, W.},
       title={Stable and unstable ideal plane flows},
        date={2002},
        ISSN={0252-9599,1860-6261},
     journal={Chinese Ann. Math. Ser. B},
      volume={23},
      number={2},
       pages={149\ndash 164},
         url={https://doi.org/10.1142/S0252959902000158},
        note={Dedicated to the memory of Jacques-Louis Lions},
}

\bib{Bardos2013}{article}{
      author={Bardos, Claude~W.},
      author={Titi, Edriss~S.},
       title={Mathematics and turbulence: where do we stand?},
        date={2013},
     journal={J. Turbul.},
      volume={14},
      number={3},
       pages={42\ndash 76},
         url={https://doi.org/10.1080/14685248.2013.771838},
}

\bib{Batchelor1967}{book}{
      author={Batchelor, G.~K.},
       title={An introduction to fluid dynamics},
   publisher={Cambridge University Press},
     address={Cambridge},
        date={1967},
}

\bib{BedrossianHeIyerWang2024}{misc}{
      author={Bedrossian, Jacob},
      author={He, Siming},
      author={Iyer, Sameer},
      author={Wang, Fei},
       title={Uniform inviscid damping and inviscid limit of the {2D}
  {Navier--Stokes} equation with {Navier} boundary conditions},
         how={arXiv preprint arXiv:2405.19249},
        date={2024},
         url={https://arxiv.org/abs/2405.19249v1},
        note={Version 1, 29 May 2024},
}

\bib{BelenkayaFriedlanderYudovich1999}{article}{
      author={Belenkaya, L.},
      author={Friedlander, Susan},
      author={Yudovich, V.},
       title={The unstable spectrum of oscillating shear flows},
        date={1999},
     journal={SIAM Journal on Applied Mathematics},
      volume={59},
      number={5},
       pages={1701\ndash 1715},
}

\bib{BonaWu2002}{article}{
      author={Bona, Jerry~L.},
      author={Wu, Jiahong},
       title={The zero-viscosity limit of the {2D} {Navier--Stokes} equations},
        date={2002},
     journal={Studies in Applied Mathematics},
      volume={109},
      number={4},
       pages={265\ndash 278},
         url={https://doi.org/10.1111/1467-9590.t01-1-00223},
}

\bib{CaflischSammartino2006}{inproceedings}{
      author={Caflisch, R.~E.},
      author={Sammartino, M.},
       title={Vortex layers in the small viscosity limit},
        date={2006},
   booktitle={Waves and stability in continuous media: Proceedings of the 13th
  conference on {WASCOM} 2005},
      editor={Monaco, Roberto},
      editor={Mulone, Giuseppe},
      editor={Rionero, Salvatore},
      editor={Ruggeri, Tommaso},
   publisher={World Scientific},
     address={Hackensack, NJ},
       pages={59\ndash 70},
         url={https://doi.org/10.1142/9789812773616_0010},
}

\bib{Chemin1996}{article}{
      author={Chemin, Jean-Yves},
       title={A remark on the inviscid limit for two-dimensional incompressible
  fluids},
        date={1996},
     journal={Communications in Partial Differential Equations},
      volume={21},
      number={11--12},
       pages={1771\ndash 1779},
}

\bib{ConstantinDrivasElgindi2022}{article}{
      author={Constantin, Peter},
      author={Drivas, Theodore~D.},
      author={Elgindi, Tarek~M.},
       title={Inviscid limit of vorticity distributions in the {Yudovich}
  class},
        date={2022},
     journal={Communications on Pure and Applied Mathematics},
      volume={75},
      number={1},
       pages={60\ndash 82},
}

\bib{Constantin1988navier}{book}{
      author={Constantin, Peter},
      author={Foia{\c{s}}, Ciprian},
       title={{N}avier--{S}tokes equations},
   publisher={University of Chicago press},
        date={1988},
}

\bib{ConstantinKukavicaVicol2015}{article}{
      author={Constantin, Peter},
      author={Kukavica, Igor},
      author={Vicol, Vlad},
       title={On the inviscid limit of the {Navier--Stokes} equations},
        date={2015},
     journal={Proceedings of the American Mathematical Society},
      volume={143},
      number={7},
       pages={3075\ndash 3090},
         url={https://doi.org/10.1090/S0002-9939-2015-12638-X},
}

\bib{coti2023stationary}{article}{
      author={Coti~Zelati, Michele},
      author={Elgindi, Tarek~M},
      author={Widmayer, Klaus},
       title={Stationary structures near the {K}olmogorov and poiseuille flows
  in the 2{D} {E}uler equations},
        date={2023},
     journal={Archive for Rational Mechanics and Analysis},
      volume={247},
      number={1},
       pages={12},
}

\bib{elgindi2026flexibility}{article}{
      author={Elgindi, Tarek~M},
      author={Huang, Yupei},
       title={On the flexibility of 2{D} {Euler} steady states},
        date={2026},
     journal={arXiv preprint arXiv:2604.12962},
}

\bib{filho2005}{article}{
      author={Filho, M. C.~Lopes},
      author={Lopes, H. J.~Nussenzveig},
      author={Planas, G.},
       title={On the inviscid limit for two-dimensional incompressible flow
  with navier friction condition},
        date={2005},
     journal={SIAM Journal on Mathematical Analysis},
      volume={36},
      number={4},
       pages={1130\ndash 1141},
         url={https://doi.org/10.1137/S0036141003432341},
}

\bib{friedlander1998instability}{article}{
      author={Friedlander, Susan},
      author={Howard, Louis},
       title={Instability in parallel flows revisited},
        date={1998},
     journal={Studies in Applied Mathematics},
      volume={101},
      number={1},
       pages={1\ndash 21},
}

\bib{FriedlanderStraussVishik1997}{article}{
      author={Friedlander, Susan},
      author={Strauss, Walter},
      author={Vishik, Misha},
       title={Nonlinear instability in an ideal fluid},
        date={1997},
     journal={Annales de l'Institut Henri Poincar{\'e} C, Analyse non
  lin{\'e}aire},
      volume={14},
      number={2},
       pages={187\ndash 209},
}

\bib{FriedlanderStraussVishik1999}{article}{
      author={Friedlander, Susan},
      author={Strauss, Walter},
      author={Vishik, Misha},
       title={Robustness of instability for the two-dimensional {Euler}
  equations},
        date={1999},
     journal={SIAM Journal on Mathematical Analysis},
      volume={30},
      number={6},
       pages={1343\ndash 1354},
}

\bib{GerardVaretMaekawaMasmoudi2018}{article}{
      author={G{\'e}rard-Varet, David},
      author={Maekawa, Yasunori},
      author={Masmoudi, Nader},
       title={Gevrey stability of {Prandtl} expansions for 2-dimensional
  {Navier--Stokes} flows},
        date={2018},
     journal={Duke Mathematical Journal},
      volume={167},
      number={13},
       pages={2531\ndash 2631},
         url={https://doi.org/10.1215/00127094-2018-0020},
}

\bib{GerardVaretNguyen2012}{article}{
      author={G{\'e}rard-Varet, David},
      author={Nguyen, Toan},
       title={Remarks on the ill-posedness of the {Prandtl} equation},
        date={2012},
     journal={Asymptotic Analysis},
      volume={77},
      number={1-2},
       pages={71\ndash 88},
         url={https://doi.org/10.3233/ASY-2011-1075},
}

\bib{GieEtAl2019}{article}{
      author={Gie, Gung-Min},
      author={Kelliher, James~P.},
      author={Lopes~Filho, Milton~C.},
      author={Mazzucato, Anna~L.},
      author={Nussenzveig~Lopes, Helena~J.},
       title={The vanishing viscosity limit for some symmetric flows},
        date={2019},
     journal={Annales de l'Institut Henri Poincar{\'e} C, Analyse non
  lin{\'e}aire},
      volume={36},
      number={5},
       pages={1237\ndash 1280},
         url={https://doi.org/10.1016/j.anihpc.2018.11.006},
}

\bib{grenier2000nonlinear}{article}{
      author={Grenier, Emmanuel},
       title={On the nonlinear instability of {Euler} and {Prandtl} equations},
        date={2000},
     journal={Communications on Pure and Applied Mathematics},
      volume={53},
      number={9},
       pages={1067\ndash 1091},
}

\bib{GrenierGuoNguyen2016}{article}{
      author={Grenier, Emmanuel},
      author={Guo, Yan},
      author={Nguyen, Toan~T.},
       title={Spectral instability of characteristic boundary layer flows},
        date={2016},
     journal={Duke Mathematical Journal},
      volume={165},
      number={16},
       pages={3085\ndash 3146},
         url={https://doi.org/10.1215/00127094-3645437},
}

\bib{kato1995perturbation}{book}{
      author={Kato, T.},
       title={Perturbation theory for linear operators},
      series={Classics in Mathematics},
   publisher={Springer Berlin Heidelberg},
        date={1995},
        ISBN={9783540586616},
         url={https://books.google.com/books?id=8ji2kN_D3BwC},
}

\bib{Kato1984}{incollection}{
      author={Kato, Tosio},
       title={Remarks on zero viscosity limit for nonstationary
  {N}avier--{S}tokes flows with boundary},
        date={1984},
   booktitle={Seminar on nonlinear partial differential equations},
      editor={Chern, S.~S.},
      series={Mathematical Sciences Research Institute Publications},
      volume={2},
   publisher={Springer},
     address={New York, NY},
       pages={85\ndash 98},
}

\bib{Kelliher2006Navier}{article}{
      author={Kelliher, James~P.},
       title={{Navier--Stokes} equations with {Navier} boundary conditions for
  a bounded domain in the plane},
        date={2006},
     journal={SIAM Journal on Mathematical Analysis},
      volume={38},
      number={1},
       pages={210\ndash 232},
         url={https://doi.org/10.1137/040612336},
}

\bib{Kelliher2007Kato}{article}{
      author={Kelliher, James~P.},
       title={On {Kato}'s conditions for vanishing viscosity},
        date={2007},
     journal={Indiana University Mathematics Journal},
      volume={56},
      number={4},
       pages={1711\ndash 1721},
         url={https://doi.org/10.1512/iumj.2007.56.3080},
}

\bib{KiselevSverak2014}{article}{
      author={Kiselev, Alexander},
      author={{\v S}ver{\'a}k, Vladimir},
       title={Small scale creation for solutions of the incompressible
  two-dimensional {E}uler equation},
        date={2014},
     journal={Annals of Mathematics},
      volume={180},
      number={3},
       pages={1205\ndash 1220},
}

\bib{KukavicaVicolWang2017}{article}{
      author={Kukavica, Igor},
      author={Vicol, Vlad},
      author={Wang, Fei},
       title={The {van Dommelen} and {Shen} singularity in the {Prandtl}
  equations},
        date={2017},
     journal={Advances in Mathematics},
      volume={307},
       pages={288\ndash 311},
         url={https://doi.org/10.1016/j.aim.2016.11.013},
}

\bib{LacaveMazzucato2016}{article}{
      author={Lacave, Christophe},
      author={Mazzucato, Anna~L.},
       title={The vanishing viscosity limit in the presence of a porous
  medium},
        date={2016},
     journal={Mathematische Annalen},
      volume={365},
      number={3-4},
       pages={1527\ndash 1557},
         url={https://doi.org/10.1007/s00208-015-1313-x},
}

\bib{LandauLifshitz1987}{book}{
      author={Landau, L.~D.},
      author={Lifshitz, E.~M.},
       title={Fluid mechanics},
     edition={2},
      series={Course of Theoretical Physics},
   publisher={Pergamon Press},
     address={Oxford},
        date={1987},
      volume={6},
}

\bib{LeiShi2016}{article}{
      author={Lei, Zhen},
      author={Shi, Jia},
       title={Infinite-time exponential growth of the {E}uler equation on
  two-dimensional torus},
        date={2016},
     journal={arXiv preprint arXiv:1608.07010},
         url={https://arxiv.org/abs/1608.07010},
}

\bib{Lin2003}{article}{
      author={Lin, Zhiwu},
       title={Instability of some ideal plane flows},
        date={2003},
     journal={SIAM Journal on Mathematical Analysis},
      volume={35},
      number={2},
       pages={318\ndash 356},
}

\bib{lin2004nonlinear}{article}{
      author={Lin, Zhiwu},
       title={Nonlinear instability of ideal plane flows},
        date={2004},
     journal={International Mathematics Research Notices},
      volume={2004},
      number={41},
       pages={2147\ndash 2178},
}

\bib{lin2011inviscid}{article}{
      author={Lin, Zhiwu},
      author={Zeng, Chongchun},
       title={Inviscid dynamical structures near couette flow},
        date={2011},
     journal={Archive for rational mechanics and analysis},
      volume={200},
      number={3},
       pages={1075\ndash 1097},
}

\bib{lin2013unstable}{article}{
      author={Lin, Zhiwu},
      author={Zeng, Chongchun},
       title={Unstable manifolds of {E}uler equations},
        date={2013},
     journal={Communications on Pure and Applied Mathematics},
      volume={66},
      number={11},
       pages={1803\ndash 1836},
}

\bib{LinZeng2022}{article}{
      author={Lin, Zhiwu},
      author={Zeng, Chongchun},
       title={Instability, index theorem, and exponential trichotomy for linear
  {Hamiltonian PDEs}},
        date={2022},
     journal={Memoirs of the American Mathematical Society},
      volume={275},
      number={1347},
}

\bib{Lions1996}{book}{
      author={Lions, Pierre-Louis},
       title={Mathematical topics in fluid mechanics. vol.~1: Incompressible
  models},
      series={Oxford Lecture Series in Mathematics and its Applications},
   publisher={The Clarendon Press, Oxford University Press},
     address={New York},
        date={1996},
      number={3},
        ISBN={9780198514879},
  url={https://global.oup.com/academic/product/mathematical-topics-in-fluid-mechanics-9780198514879},
}

\bib{LiuZheng1999}{book}{
      author={Liu, Zhuangyi},
      author={Zheng, Songmu},
       title={Semigroups associated with dissipative systems},
      series={Research Notes in Mathematics},
   publisher={Chapman \& Hall/CRC},
     address={Boca Raton, FL},
        date={1999},
      volume={398},
}

\bib{Maekawa2014}{article}{
      author={Maekawa, Yasunori},
       title={On the inviscid limit problem of the vorticity equations for
  viscous incompressible flows in the half-plane},
        date={2014},
     journal={Communications on Pure and Applied Mathematics},
      volume={67},
      number={7},
       pages={1045\ndash 1128},
         url={https://doi.org/10.1002/cpa.21516},
}

\bib{maekawa2018}{article}{
      author={Maekawa, Yasunori},
      author={Mazzucato, Anna~L.},
       title={The inviscid limit and boundary layers for {N}avier--{S}tokes
  flows},
        date={2018},
       pages={781\ndash 828},
}

\bib{Masmoudi2007remarks}{article}{
      author={Masmoudi, Nader},
       title={Remarks about the inviscid limit of the {N}avier--{S}tokes
  system},
        date={2007},
     journal={Communications in mathematical physics},
      volume={270},
      number={3},
       pages={777\ndash 788},
}

\bib{MazzucatoTaylor2008}{article}{
      author={Mazzucato, Anna~L.},
      author={Taylor, Michael~E.},
       title={Vanishing viscosity plane parallel channel flow and related
  singular perturbation problems},
        date={2008},
     journal={Analysis \& PDE},
      volume={1},
      number={1},
       pages={35\ndash 93},
         url={https://doi.org/10.2140/apde.2008.1.35},
}

\bib{MeshalkinSinai1961}{article}{
      author={Meshalkin, L.~D.},
      author={Sinai, Ya.~G.},
       title={Investigation of the stability of a stationary solution of a
  system of equations for the plane movement of an incompressible viscous
  liquid},
        date={1961},
     journal={Journal of Applied Mathematics and Mechanics},
      volume={25},
      number={6},
       pages={1700\ndash 1705},
}

\bib{SammartinoCaflisch1998II}{article}{
      author={Sammartino, Marco},
      author={Caflisch, Russel~E.},
       title={Zero viscosity limit for analytic solutions of the
  {Navier--Stokes} equation on a half-space. {II}. construction of the
  {Navier--Stokes} solution},
        date={1998},
     journal={Communications in Mathematical Physics},
      volume={192},
       pages={463\ndash 491},
         url={https://doi.org/10.1007/s002200050305},
}

\bib{Seis2021}{article}{
      author={Seis, Christian},
       title={A note on the vanishing viscosity limit in the {Yudovich} class},
        date={2021},
     journal={Canadian Mathematical Bulletin},
      volume={64},
      number={1},
       pages={112\ndash 122},
}

\bib{ShvydkoyFriedlander2008}{article}{
      author={Shvydkoy, Roman},
      author={Friedlander, Susan},
       title={The unstable spectrum of the {N}avier--{S}tokes operator in the
  limit of vanishing viscosity},
        date={2008},
     journal={Ann. Inst. H. Poincar\'e C Anal. Non Lin\'eaire},
      volume={25},
      number={4},
       pages={713\ndash 724},
}

\bib{ShvydkoyLatushkin2005}{article}{
      author={Shvydkoy, Roman},
      author={Latushkin, Yuri},
       title={Essential spectrum of the linearized {2D Euler} equation and
  {Lyapunov--Oseledets} exponents},
        date={2005},
     journal={Journal of Mathematical Fluid Mechanics},
      volume={7},
      number={2},
       pages={164\ndash 178},
         url={https://doi.org/10.1007/s00021-004-0114-x},
}

\bib{vasseur2023}{article}{
      author={Vasseur, Alexis~F.},
      author={Yang, Jincheng},
       title={Boundary vorticity estimates for {N}avier--{S}tokes and
  application to the inviscid limit},
        date={2023},
     journal={SIAM Journal on Mathematical Analysis},
      volume={55},
      number={4},
       pages={3081\ndash 3107},
}

\bib{vasseur2024}{article}{
      author={Vasseur, Alexis~F.},
      author={Yang, Jincheng},
       title={Layer separation of the 3{D} incompressible {N}avier--{S}tokes
  equation in a bounded domain},
        date={2024},
     journal={Communications in Partial Differential Equations},
      volume={49},
      number={4},
       pages={381\ndash 409},
}

\bib{Vishik2018PartI}{article}{
      author={Vishik, Misha},
       title={Instability and non-uniqueness in the {C}auchy problem for the
  {E}uler equations of an ideal incompressible fluid. {P}art {I}},
        date={2018},
     journal={arXiv preprint arXiv:1805.09426},
}

\bib{Vishik2018PartII}{article}{
      author={Vishik, Misha},
       title={Instability and non-uniqueness in the {C}auchy problem for the
  {E}uler equations of an ideal incompressible fluid. {P}art {II}},
        date={2018},
     journal={arXiv preprint arXiv:1805.09440},
}

\end{biblist}
\end{bibdiv}
\end{document}